\documentclass{amsart}

\usepackage{kantlipsum}
\usepackage[shortlabels]{enumitem}

\usepackage{amsmath,amscd}
\usepackage{amssymb}
\usepackage{mathtools}

\usepackage[all]{xy}

\usepackage[T1]{fontenc}

\usepackage{times}

\usepackage[T1]{fontenc}
\usepackage{graphicx}
\usepackage[svgnames]{xcolor}
\usepackage{framed}

\usepackage{tikz}

\usepackage[driverfallback=dvipdfm]{hyperref}

\author{Liran Shaul}
\address{Department of Algebra, Faculty of Mathematics and Physics, Charles University in Prague, Sokolovsk\'a 83, 186 75 Praha, Czech Republic}

\email{shaul@karlin.mff.cuni.cz}
\newtheorem{thm}[equation]{Theorem}
\newtheorem*{thm*}{Theorem}

\newtheorem{cthm}{Classical Theorem}

\newtheorem{cor}[equation]{Corollary}
\newtheorem{prop}[equation]{Proposition}
\newtheorem{lem}[equation]{Lemma}

\theoremstyle{definition}
\newtheorem{dfn}[equation]{Definition}
\newtheorem{rem}[equation]{Remark}

\newtheorem{exa}[equation]{Example}

\newcommand{\inj}{\hookrightarrow}

\newcommand{\opn}{\operatorname}
\newcommand{\cat}[1]{\operatorname{\mathsf{#1}}}

\newcommand{\mfrak}[1]{\mathfrak{#1}}

\newcommand{\mrm}[1]{\mathrm{#1}}

\renewcommand{\k}{\Bbbk}

\renewcommand{\a}{\mfrak{a}}
\renewcommand{\b}{\mfrak{b}}
\renewcommand{\c}{\mfrak{c}}

\newcommand{\m}{\mfrak{m}}
\newcommand{\n}{\mfrak{n}}
\newcommand{\p}{\mfrak{p}}
\newcommand{\q}{\mfrak{q}}
\newcommand{\injdim}{\operatorname{inj\,dim}}

\newcommand{\flatdim}{\operatorname{flat\,dim}}
\newcommand{\depth}{\operatorname{depth}}
\newcommand{\amp}{\operatorname{amp}}
\def\skewtimes{\ltimes\!}
\newcommand{\lcdim}{\operatorname{lc.\dim}}

\begin{document}

\title{The Cohen-Macaulay property in derived commutative algebra}

\begin{abstract}
By extending some basic results of Grothendieck and Foxby about local cohomology to commutative DG-rings,
we prove new amplitude inequalities about finite DG-modules of finite injective dimension over commutative local DG-rings,
complementing results of J{\o}rgensen and resolving a recent conjecture of Minamoto.
When these inequalities are equalities, we arrive to the notion of a local-Cohen-Macaulay DG-ring.
We make a detailed study of this notion,
showing that much of the classical theory of Cohen-Macaulay rings and modules can be generalized to the derived setting,
and that there are many natural examples of local-Cohen-Macaulay DG-rings.
In particular, local Gorenstein DG-rings are local-Cohen-Macaulay.
Our work is in a non-positive cohomological situation,
allowing the Cohen-Macaulay condition to be introduced to derived algebraic geometry,
but we also discuss extensions of it to non-negative DG-rings,
which could lead to the concept of Cohen-Macaulayness in topology.
\end{abstract}


\setcounter{tocdepth}{1}
\setcounter{section}{-1}

\maketitle
\tableofcontents

\section{Introduction}
 
In classical commutative algebra,
the classes of Gorenstein and Cohen-Macaulay rings are among the most important classes of local rings.
In particular, the theory of Cohen-Macaulay rings and modules is among the most deep and influential parts of commutative algebra,
with numerous applications in commutative algebra, algebraic geometry and combinatorics.

The Gorenstein condition has been introduced long ago to higher algebra and related fields.
Its first incarnation was probably in the work \cite{FHT} of F{\'e}lix, Halperin and Thomas about Gorenstein spaces in topology.
Some other occurrences of it are in the works of Avramov and Foxby \cite{AF1} and Frankild, Iyengar and J{\o}rgensen \cite{FJ,FIJ}
about Gorenstein DG-rings, of Dwyer, Greenlees and Iyengar \cite{DGI} about Gorenstein $\mathbb{S}$-algebras (where $\mathbb{S}$ is the sphere spectrum), in the work of Lurie about Gorenstein spectral algebraic spaces \cite[Chapter 6.6.5]{Lu} and many more.

Despite the great success of the Gorenstein condition in higher algebra,
and of the Cohen-Macaulay condition in classical commutative algebra,
until now it was completely missing from higher algebra.
The aim of this paper is to extend the theory of Cohen-Macaulay rings and Cohen-Macaulay modules to the setting of commutative noetherian differential graded rings.

We work with commutative non-positive DG-rings $A = \bigoplus_{n=-\infty}^0 A^n$ with a differential of degree $+1$.
These include (and in characteristic zero are equivalent to) the normalizations of the simplicial commutative rings,
so they include affine derived schemes.

Given a commutative DG-ring (or a ring) $A$,
we denote by $\cat{D}(A)$ the unbounded derived category of $A$-modules.
For $M \in \cat{D}(A)$,
its amplitude is the number (or $+\infty$)
\[
\amp(M) = \sup\{i\mid \mrm{H}^i(M) \ne 0\} - \inf\{i\mid \mrm{H}^i(M) \ne 0\}.
\]

For a ring $A$, we denote by $\dim(A)$ the Krull dimension of $A$,
and similarly for an $A$-module $M$, $\dim(M)$ is the Krull dimension of $M$.
To describe the main results of this paper, 
let us first summarize some important facts from the classical theory of Cohen-Macaulay rings which we are going to generalize:

\begin{cthm}\label{cthmA}
Let $(A,\m)$ be a noetherian local ring.
Then the following are equivalent:
\begin{enumerate}
\item The ring $A$ is Cohen-Macaulay.
\item There is an $A$-regular sequence $x_1,\dots,x_d \in \m \subseteq A$ of length $d=\dim(A)$ 
which is a system of parameters of $A$,
and moreover, for each $1 \le i \le d$, the ring $A_i = A/(x_1,\dots,x_i)$ is a Cohen-Macaulay ring and $\dim(A_i) = \dim(A) - i$.
\item One has $\amp\left(\mrm{R}\Gamma_{\m}(A)\right) = 0$, i.e, the local cohomology of $A$ is concentrated in a single degree.
\item The $\m$-adic completion $(\widehat{A},\widehat{\m})$ is a Cohen-Macaulay ring.
\item The Bass conjecture holds: there exists a finitely generated $A$-module $M \ne 0$ of finite injective dimension;
that is, $0 \ne M \in \cat{D}^{\mrm{b}}_{\mrm{f}}(A)$ such that $\amp(M) = 0$ and $\injdim_A(M) < \infty$.
\end{enumerate}
If moreover $A$ has a dualizing complex $R$ then this is also equivalent to:
\begin{enumerate}
\setcounter{enumi}{5}
\item One has $\amp\left(R\right) = 0$, i.e, $A$ has a dualizing module.
\end{enumerate}
Furthermore, the following rings satisfy these equivalent statements:
\begin{enumerate}[(a)]
\item Gorenstein rings.
\item Local rings $A$ with $\dim(A) = 0$.
\end{enumerate}
\end{cthm}

A few remarks are in order. 
First, that item (5) holds is the proof of the Bass conjecture, 
which was introduced in \cite{Ba},
due to Peskine and Szpiro (see \cite{PS}).
A noetherian local ring has a dualizing complex if and only if it a quotient of a Gorenstein ring.
This is a theorem of Kawasaki (see \cite{Ka}), proving a conjecture of Sharp. 
The fact that a local ring with a dualizing complex is Cohen-Macaulay if and only if it has a dualizing module is a consequence of Grothendieck's local duality theorem.

Next, let us recall some basics of the theory of Cohen-Macaulay modules over Cohen-Macaulay rings which we will generalize:

\begin{cthm}\label{cthmB}
Let $(A,\m)$ be a noetherian local Cohen-Macaulay ring with a dualizing module $R$.
Denote by $\cat{CM}(A)$ the category of Cohen-Macaulay $A$-modules,
and by $\cat{MCM}(A)$ its full subcategory of maximal Cohen-Macaulay $A$-modules.
Then the following hold:
\begin{itemize}
\item
The functor $D(-) := \mrm{R}\opn{Hom}_A(-,R)$ induces a duality on $\cat{CM}(A)$.
More precisely, if $M \in \cat{CM}(A)$, then $D(M)$ is a shift of an object of $\cat{CM}(A)$,
and the natural map $M \to D(D(M))$ is an isomorphism.
\item 
The above duality restricts to a duality on $\cat{MCM}(A)$, so that
if $M\in \cat{MCM}(A)$, then its dual $D(M)$ is a shift of an object in $\cat{MCM}(A)$.
\end{itemize}
Moreover, we have that $A,R \in \cat{MCM}(A)$.
\end{cthm}

We wish to generalize these results to derived commutative algebra.
We say that a DG-ring $A$ is noetherian if the ring $\mrm{H}^0(A)$ is a noetherian ring,
and for each $i <0$ the $\mrm{H}^0(A)$-module $\mrm{H}^i(A)$ is finitely generated.
If $A$ is noetherian and $(\mrm{H}^0(A),\bar{\m},\k)$ is a local ring,
we say that $(A,\bar{\m})$ (or $(A,\bar{\m},\k)$) is a noetherian local DG-ring.
We will recall in Section \ref{sec:lc} the notion of local cohomology
of a DG-ring $A$ with respect to a finitely generated ideal in $\mrm{H}^0(A)$.
In particular, if $(A,\bar{\m})$ is a noetherian local DG-ring,
attached to it is the local cohomology functor
\[
\mrm{R}\Gamma_{\bar{\m}}:\cat{D}(A) \to \cat{D}(A).
\]
The analogue of the notion of a dualizing complex over a DG-ring is called a dualizing DG-module,
and is recalled in Section \ref{sec:dc}.
The notion of a regular sequence in the DG-setting is recalled in Definition \ref{dfn:seq}.

As a first step to generalize Classical Theorem \ref{cthmA}, 
we prove the following new inequalities about the amplitude of local cohomology and of dualizing DG-modules,
and on the length of regular sequences over noetherian local DG-rings:

\begin{thm}\label{thm:main-bounds}
The following inequalities hold:
\begin{enumerate}[wide, labelwidth=!, labelindent=0pt]
\item If $(A,\bar{\m})$ is a noetherian local DG-ring with bounded cohomology then
\[
\amp(A) \le \amp\left(\mrm{R}\Gamma_{\bar{\m}}(A)\right).
\]
\item If $(A,\bar{\m})$ is a noetherian local DG-ring with bounded cohomology,
and $\opn{seq.depth}(A)$ denotes the maximal length of an $A$-regular sequence contained in $\bar{\m}$,
then 
\[
\opn{seq.depth}(A) \le \dim(\mrm{H}^0(A)).
\]
\item If $A$ is a noetherian DG-ring, and $R$ is a dualizing DG-module over $A$
then 
\[
\amp(A) \le \amp(R).
\]
\end{enumerate}
\end{thm}

This result is contained in Theorem \ref{thm:main-amp} and Corollary \ref{cor:chdepth} below.
It is worth noting that item (2) above is non-trivial: Given a noetherian local DG-ring $(A,\bar{\m})$,
and given $\bar{x} \in \bar{\m}$, the noetherian local DG-ring $A//\bar{x}$ is given by a Koszul-complex type construction which is recalled in the beginning of Section \ref{sec:regseq}. 
Unlike rings, the situation for DG-rings is that even if $\bar{x} \in \bar{\m}$ is $A$-regular,
it could happen that $\dim(\mrm{H}^0(A//\bar{x})) = \dim(\mrm{H}^0(A))$ (see Example \ref{exa:reg-not-par}).

Given a commutative DG-ring $A$, and a finitely generated ideal $\bar{\a} \subseteq \mrm{H}^0(A)$,
the derived $\bar{\a}$-adic completion of $A$, 
denoted by $\mrm{L}\Lambda(A,\bar{\a})$, is a commutative DG-ring,
defined in \cite{Sh1}, and recalled in Section \ref{sec:DerComp} below.

In view of the inequalities in Theorem \ref{thm:main-bounds},
it is natural to study DG-rings for which these are equalities.
Let us say that a noetherian local DG-ring $(A,\bar{\m})$ with bounded cohomology is \textbf{local-Cohen-Macaulay} if there is an equality 
$\opn{seq.depth}(A) = \dim(\mrm{H}^0(A))$.
We characterize local-Cohen-Macaulay DG-rings in the next result which is a precise derived analogue of Classical Theorem \ref{cthmA}.

\begin{thm}\label{thm:CMMain}
Let $(A,\bar{\m})$ be a noetherian local DG-ring with bounded cohomology.
Then the following are equivalent:
\begin{enumerate}
\item The DG-ring $A$ is local-Cohen-Macaulay, i.e, $\opn{seq.depth}(A) = \dim(\mrm{H}^0(A))$.
\item There exists an $A$-regular sequence $\bar{x}_1,\dots,\bar{x}_d \in \bar{\m}\subseteq \mrm{H}^0(A)$ of length 
$d=\dim(\mrm{H}^0(A))$ which is a system of parameters of $\mrm{H}^0(A)$,
and moreover, for each $1 \le i \le d$, 
the DG-ring $A_i = A//(\bar{x}_1,\dots,\bar{x}_i)$ is local-Cohen-Macaulay,
and there is an equality
$\dim(\mrm{H}^0(A_i)) = \dim(\mrm{H}^0(A)) - i$.
\item There is an equality $\amp\left(\mrm{R}\Gamma_{\bar{\m}}(A)\right) = \amp(A)$.
\item The derived $\bar{\m}$-adic completion $\mrm{L}\Lambda(A,\bar{\m})$ is local-Cohen-Macaulay.
\item The analogue of the Bass conjecture holds: 
there exists $0 \ne M \in \cat{D}^{\mrm{b}}_{\mrm{f}}(A)$ such that $\amp(M) = \amp(A)$,
and $\injdim_A(M) < \infty$.\footnote{To be precise, we show that the other conditions imply this condition, and that the converse holds under the additional assumption that $A$ has a noetherian model, see Remark \ref{rem:noet-model} for details why this extra assumption is needed.}
\end{enumerate}
If moreover $A$ has a dualizing DG-module $R$
then this is also equivalent to:
\begin{enumerate}
\setcounter{enumi}{5}
\item There is an equality $\amp\left(R\right) = \amp(A)$.
\end{enumerate}
Furthermore, the following DG-rings are local-Cohen-Macaulay:
\begin{enumerate}[(a)]
\item Local Gorenstein DG-rings.
\item Local DG-rings $(A,\bar{\m})$ with $\dim(\mrm{H}^0(A)) = 0$.
\end{enumerate}
\end{thm}

The proof of this result takes the majority of Sections \ref{sec:CMDG} and \ref{sec:regseq} below.
The reason for the terminology \textbf{local-Cohen-Macaulay} is that, 
unlike the case of rings, this property need not be preserved under localization.
See Section \ref{sec:localiz} below for a discussion and for a global variant of this property.

Next, we study local-Cohen-Macaulay DG-modules over a local-Cohen-Macaulay DG-ring, and prove an analogue of Classical Theorem \ref{cthmB}.
We give in Section \ref{sec:MCM} below definitions of local-Cohen-Macaulay DG-modules and maximal local-Cohen-Macaulay DG-modules, and show in Section \ref{sec:MCM} that:

\begin{thm}
Let $(A,\bar{\m})$ be a noetherian local-Cohen-Macaulay DG-ring,
and let $R$ be a dualizing DG-module over $A$.
Denote by $\cat{CM}(A)$ the category of local-Cohen-Macaulay DG-modules over $A$.
Let $\cat{MCM}(A)$ be the full subcategory of $\cat{CM}(A)$ which consists of
maximal local-Cohen-Macaulay DG-modules over $A$.
Then the following hold:
\begin{itemize}
\item
The functor $\mrm{R}\opn{Hom}_A(-,R)$ induces a duality on $\cat{CM}(A)$.
\item 
The above duality restricts to a duality on $\cat{MCM}(A)$, so that
if $M\in \cat{MCM}(A)$, then its dual $\mrm{R}\opn{Hom}_A(M,R)$ is an object in $\cat{MCM}(A)$.
\end{itemize}
Moreover, we have that $A,R \in \cat{MCM}(A)$.
\end{thm}

Let us now describe the rest of the contents of this paper.
In Section \ref{sec:prel} we gather various preliminaries about DG-rings that will be used throughout this paper.
In sections \ref{sec:KDim} and \ref{sec:depth} we make a detailed study of local cohomology in the DG setting.
The main result is Theorem \ref{thm:non-vanish} which is a DG version of Grothendieck's vanishing and non-vanishing theorems for local cohomology.
We introduce the notion of a local-Cohen-Macaulay DG-ring in Section \ref{sec:CMDG},
give examples, and study some its basic properties.
Then, in Section \ref{sec:regseq} we study regular sequences, associated primes and other related notions in the DG-setting,
following works of Christensen and Minamoto.
Using these ideas and our results about local cohomology, 
we show in Corollary \ref{cor:DGhasSOP} that:

\begin{thm}
Let $(A,\bar{\m})$ be a noetherian local DG-ring with bounded cohomology.
Then there exists a maximal $A$-regular sequence $\bar{x}_1,\dots,\bar{x}_n \in \bar{\m}$
such that $\bar{x}_1,\dots,\bar{x}_n$ can be completed to system of parameters of $\mrm{H}^0(A)$.
\end{thm}

Using this result, 
we then prove in Theorem \ref{thm:bassConj} a DG-version of the Bass conjecture,
and a bit more generally:

\begin{thm}
Let $(A,\bar{\m})$ be a noetherian local DG-ring with bounded cohomology.
\begin{enumerate}[wide, labelwidth=!, labelindent=0pt]
\item If $A$ is local-Cohen-Macaulay, there exists $0 \ncong M \in \cat{D}^{\mrm{b}}_{\mrm{f}}(A)$
such that $\injdim_A(M) < \infty$, and such that $\amp(M) = \amp(A)$.
\item Assume further that $A$ has a noetherian model.
For any $0 \ncong M \in \cat{D}^{\mrm{b}}_{\mrm{f}}(A)$
such that $\injdim_A(M) < \infty$,
we have that $\amp(M) \ge \amp(A)$.
If there exists such $M$ with $\amp(M) = \amp(A)$,
then $A$ is local-Cohen-Macaulay.
\end{enumerate}
\end{thm}

This amplitude inequality, which was mentioned in the abstract,
solves a recent conjecture of Minamoto.

Section \ref{sec:MCM} introduces local-Cohen-Macaulay and maximal local-Cohen-Macaulay DG-modules over a local DG-ring.
Among its results, we prove the following general result about the structure of dualizing DG-modules over local DG-rings.

\begin{thm}
Let $(A,\bar{\m})$ be a noetherian local DG-ring with bounded cohomology.
Setting $n = \amp(A)$ and $d = \dim(\mrm{H}^0(A))$,
let $R$ be a dualizing DG-module over $A$,
normalized so that $\inf(R) = -d$.
Then the following hold:
\begin{enumerate}
\item
For every $0 \le i \le d$ there is an inequality
\[
\dim\left(\mrm{H}^{-i+n}(R)\right) \le i
\]
\item We have that $A$ is a local-Cohen-Macaulay DG-ring if and only if 
\[
\dim\left(\mrm{H}^{\sup(R)}(R)\right) = d.
\]
\end{enumerate}
\end{thm}
This result is contained in Theorem \ref{thm:str-of-dc}.

In Section \ref{sec:text} we consider the problem of determining
when are DG-rings that arise from trivial extensions of local rings by cochain complexes are Cohen-Macaulay.
We show in particular that any Cohen-Macaulay module over a local ring give rise to a local-Cohen-Macaulay DG-ring.

Section \ref{sec:localiz} discusses two points where the DG theory diverges from the classical theory:
independence of the Cohen-Macaulay property from the base,
and localization.
We explain the reason for this divergence, 
and construct an example of a local-Cohen-Macaulay DG-ring which has a localization who is not local-Cohen-Macaulay.
We then give a global definition of the notion of a Cohen-Macaulay DG-ring, 
and prove in Corollary \ref{cor:irr-loc} that any local-Cohen-Macaulay DG-ring which has a dualizing DG-module,
and whose spectrum is irreducible,
is a Cohen-Macaulay DG-ring in the global sense.

In the final Section \ref{sec:positive} we briefly discuss the problem of defining Cohen-Macaulay DG-rings
in the case where DG-rings are non-negatively graded. 
Such DG-rings arise in topology. 
We explain that our amplitude inequalities described above do not hold in the non-negative case,
and suggest a possible way to overcome this problem.

\numberwithin{equation}{section}

\section{Preliminaries}\label{sec:prel}

In this section we will gather various preliminaries about commutative DG-rings that will be used throughout this paper.
A complete reference about derived categories of differential graded rings is the book \cite{YeBook},
and a good summary is in \cite[Section 1]{Ye2}.
However, our terminology will sometimes diverge from the terminology of \cite{YeBook},
and we will explicitly indicate such changes in terminology.

\subsection{Basics about commutative DG-rings, noetherian conditions}

A differential graded ring (abbreviated DG-ring) is a graded ring
\[
A = \bigoplus_{n=-\infty}^{\infty} A^n
\]
equipped with a $\mathbb{Z}$-linear differential $d:A \to A$ of degree $+1$,
such that the Leibniz rule
\begin{equation}\label{eqn:Leib}
d(a\cdot b) = d(a)\cdot b + (-1)^i\cdot a \cdot d(b)
\end{equation}
is satisfied for any $a \in A^i, b \in A^j$ and any $i,j \in \mathbb{Z}$.
We will further say that $A$ is commutative (called strongly commutative in \cite{YeBook}) 
if $b\cdot a = (-1)^{i\cdot j}\cdot a \cdot b$,
and moreover, if $i$ is odd, then $a^2 = 0$.
\textbf{All DG-rings in this paper will be assumed to be commutative}.
A DG-ring $A$ is called non-positive if $A^i = 0$ for all $i>0$.
From now on, in the rest of this paper except Section \ref{sec:positive}, \textbf{we will assume that all DG-rings are non-positive}.

Taking cohomology, note that $\mrm{H}^0(A)$ has the structure of a commutative ring.
We will often denote it by $\bar{A} := \mrm{H}^0(A)$. It is called the \textbf{cohomological reduction} of $A$.
Note that there is a natural map of DG-rings $\pi_A : A \to \bar{A}$.
The set $A^0$ of degree zero elements of $A$ is also a commutative ring,
and $\mrm{H}^0(A)$ is a quotient of it.

A differential graded-module $M$ over $A$ is a graded $A$-module $M$ equipped with a differential $d:M \to M$ of degree $+1$ which satisfies a Leibniz rule similar to (\ref{eqn:Leib}). 
The DG-modules over $A$ form an abelian category, denoted by $\opn{DGMod}(A)$,
in which the morphisms are given by degree $0$ $A$-linear homomorphisms which respect the differential.
Inverting quasi-isomorphisms in $\opn{DGMod}(A)$, 
we obtain the derived category of DG-modules over $A$,
denoted by $\cat{D}(A)$. It is a triangulated category.
For any $M \in \cat{D}(A)$, 
and any $n \in \mathbb{Z}$,
we have that $\mrm{H}^n(M)$ is an $\mrm{H}^0(A)$-module.

For any $n \in \mathbb{Z}$ there are smart truncation functors
\[
\opn{smt}^{>n}, \opn{smt}^{\le n} : \cat{D}(A) \to \cat{D}(A)
\]
such that for all $M \in \cat{D}(A)$, 
there are equalities
\[
\mrm{H}^i\left(\opn{smt}^{>n}(M)\right) = \left\{ \begin{array}{lr} \mrm{H}^i(M), & \text{if $i > n$,}\\
																				                             0, & \text{if $i \le n$,}\end{array}\right.
\]
and
\[
\mrm{H}^i\left(\opn{smt}^{\le n}(M)\right) = \left\{ \begin{array}{lr} \mrm{H}^i(M), & \text{if $i \le n$,}\\
																				                             0, & \text{if $i > n$.}\end{array}\right.
\]
Moreover, there is a distinguished triangle
\[
\opn{smt}^{\le n}(M) \to M \to \opn{smt}^{>n}(M) \to \opn{smt}^{\le n}(M)[1]
\]
in $\cat{D}(A)$.

Given $M \in \cat{D}(A)$, the infimum and supremum of $M$ are the numbers (or $\pm \infty$)
\[
\inf(M) = \inf\{n \in \mathbb{Z} \mid \mrm{H}^n(M) \ne 0\},
\quad
\sup(M) = \sup\{n \in \mathbb{Z} \mid \mrm{H}^n(M) \ne 0\}.
\]

In the book \cite{YeBook}, what we denote here by $\inf(M)$ (resp. $\sup(M)$) is denoted by $\inf(\mrm{H}(M))$ (resp., $\sup(\mrm{H}(M))$). We prefer the shorter notation used here, as this paper is entirely cohomological in nature,
and we will never need to consider the non-cohomological infimum and supremum.

The full subcategory of $\cat{D}(A)$ consisting of DG-modules $M$ with $\inf(M) > -\infty$ is denoted by $\cat{D}^{+}(A)$,
and the full subcategory of $\cat{D}(A)$ consisting of DG-modules $M$ with $\sup(M) < \infty$ is denoted by $\cat{D}^{-}(A)$.
These are triangulated subcategories of $\cat{D}(A)$.
Their intersection is also a triangulated subcategory of $\cat{D}(A)$,
the category of bounded DG-modules (called cohomologically bounded DG-modules in \cite{YeBook}), which will be denoted by $\cat{D}^{\mrm{b}}(A)$.

Given $M \in \cat{D}^{\mrm{b}}(A)$,
its amplitude is the number
\[
\amp(M) = \sup(M) - \inf(M) \in \mathbb{N}.
\]
If $M \notin \cat{D}^{\mrm{b}}(A)$, 
we set $\amp(M) = +\infty$. 
Again, in \cite{YeBook}, this is called the cohomological amplitude, and is denoted by $\amp(\mrm{H}(M))$.
We will say that a DG-ring $A$ has bounded cohomology if $\amp(A) < \infty$.
Under the assumption that $A$ is non-positive, 
this is equivalent to assuming that $\inf(A) > -\infty$.

A DG-ring $A$ is called noetherian (the terminology in \cite{YeBook} is cohomologically pseudo-noetherian) if the commutative ring $\mrm{H}^0(A)$ is a noetherian ring,
and for all $i < 0$, the $\mrm{H}^0(A)$-module $\mrm{H}^i(A)$ is finitely generated.
See \cite[Theorem 6.6]{Sh2} for a justification of this definition.

If $A$ is a noetherian DG-ring,
we say that $M \in \cat{D}(A)$ has finitely generated cohomology if for all $n \in \mathbb{Z}$,
the $\mrm{H}^0(A)$-modules $\mrm{H}^n(M)$ are finitely generated.
We denote by $\cat{D}_{\mrm{f}}(A)$ the full triangulated subcategory of $\cat{D}(A)$ consisting of DG-modules with finitely generated cohomology. We also set $\cat{D}^{-}_{\mrm{f}}(A) = \cat{D}_{\mrm{f}}(A) \cap \cat{D}^{-}(A)$.
Similarly we will consider $\cat{D}^{+}_{\mrm{f}}(A), \cat{D}^{\mrm{b}}_{\mrm{f}}(A)$. 
All these are full triangulated subcategories of $\cat{D}(A)$.

If $A$ is a noetherian DG-ring,
and if the noetherian ring $\mrm{H}^0(A)$ is a local ring with maximal ideal $\bar{\m}$,
we will say that $(A,\bar{\m})$ is a noetherian local DG-ring.

We say that a noetherian DG-ring $A$ has a noetherian model if there exist commutative DG-rings $B,C_1,\dots,C_n,A_1,\dots,A_n$,
and quasi-isomorphisms of DG-rings 
\[
A \leftarrow C_1 \rightarrow A_1 \leftarrow C_2 \rightarrow A_2 \leftarrow C_3 \rightarrow \dots \rightarrow A_n \leftarrow B
\]
such that $B^0$ is a noetherian ring, and for each $i<0$,
the $B^0$-module $B^i$ is finitely generated.
As shown in the proof of \cite[Lemma 7.8]{Ye1}, 
a sufficient condition for $A$ to have a noetherian model is that 
there exists a noetherian ring $\k$ and a map of DG-rings $\k \to A$,
such that the induced map $\k \to \mrm{H}^0(A)$ is essentially of finite type.
The above zig-zag of quasi-isomorphisms implies that the triangulated categories $\cat{D}(A)$
and $\cat{D}(B)$ are equivalent, and this equivalence respects standard derived functors (see \cite[Theorem 12.7.2]{YeBook}),
so with regards to statements about their derived categories,
we will be able to freely replace $A$ with $B$.

We do not know if any noetherian DG-ring has a noetherian model,
though all noetherian DG-rings that arise in nature do have a noetherian model.
Nevertheless, with the exception of one point (see Remark \ref{rem:noet-model}),
we will avoid using the noetherian model assumption in this paper,
and will almost always make only the weaker noetherian assumption, without assuming the existence of a noetherian model.

\subsection{Reduction functors}

Given a commutative DG-ring $A$,
the natural map of DG-rings $\pi_A:A \to \mrm{H}^0(A)$ induces two functors
\[
\mrm{R}\opn{Hom}_A(\mrm{H}^0(A),-) : \cat{D}^{+}(A) \to \cat{D}^{+}(A),
\]
and
\[
-\otimes^{\mrm{L}}_A \mrm{H}^0(A) : \cat{D}^{-}(A) \to \cat{D}^{-}(A)
\]
which are sometimes called the reduction functors,
and are extremely useful in studying $\cat{D}(A)$.
One reason for their usefulness is the following property:
for any $M \in \cat{D}^{+}(A)$,
by \cite[Proposition 3.3]{Sh2}
we have that
\begin{equation}\label{eqn:red-rhom}
\inf(M) = \inf\left(\mrm{R}\opn{Hom}_A(\mrm{H}^0(A),M)\right), \hspace{0.5em} \mrm{H}^{\inf(M)}(M) \cong \mrm{H}^{\inf(M)}\left(\mrm{R}\opn{Hom}_A(\mrm{H}^0(A),M)\right)
\end{equation}
Dually, for any $M \in \cat{D}^{-}(A)$,
it follows from the proof of \cite[Proposition 3.1]{Ye1} that
\begin{equation}\label{eqn:red-ten}
\sup(M) = \sup\left(M\otimes^{\mrm{L}}_A \mrm{H}^0(A)\right), \quad \mrm{H}^{\sup(M)}(M) \cong \mrm{H}^{\sup(M)}\left(M\otimes^{\mrm{L}}_A \mrm{H}^0(A)\right)
\end{equation}

\subsection{Injective DG-modules}

This section follows \cite{Sh2}.
Given a commutative DG-ring $A$,
and given $M \in \cat{D}^{+}(A)$,
the injective dimension of $M$, denoted by $\injdim_A(M)$, is defined in \cite[Section 2]{Sh2}, 
similarly to the definition of injective dimension over rings.
An injective DG-module is a DG-module $M \in \cat{D}^{+}(A)$ such that either $M \cong 0$,
or the injective dimension of $M$ is $0$, and moreover $\inf(M) =0$.
The category of injectives over $A$ is denoted by $\opn{Inj}(A)$.
The functor $\mrm{H}^0$ is an equivalence of categories
\[
\mrm{H}^0:\opn{Inj}(A) \to \opn{Inj}(\mrm{H}^0(A)).
\]
In particular, if $(A,\bar{\m})$ is a commutative noetherian local DG-ring,
there is, unique up to isomorphism, DG-module $E \in \opn{Inj}(A)$
such that $\mrm{H}^0(E)$ is the injective hull of the residue field of the local ring $(\mrm{H}^0(A),\bar{\m})$.
Following \cite[Section 7]{Sh2}, 
we will denote this DG-module by $E(A,\bar{\m})$.

\subsection{Dualizing DG-modules}\label{sec:dc}

We now recall the notion of a dualizing DG-module over a noetherian DG-ring.
These generalize Grothendieck's notion of a dualizing complex over noetherian rings (see \cite[Chapter V]{RD}).
References for all facts in this section are \cite{FIJ,Ye1}.
Let $A$ be a commutative noetherian DG-ring.
We say that a DG-module $R \in \cat{D}^{+}_{\mrm{f}}(A)$ is a dualizing DG-module over $A$
if $R$ has finite injective dimension over $A$,
and the canonical map
\[
A \to \mrm{R}\opn{Hom}_A(R,R)
\]
is an isomorphism in $\cat{D}(A)$.
It follows that if $R$ is a dualizing DG-module over $A$,
then for any $M \in \cat{D}_{\mrm{f}}(A)$,
the natural map
\[
M \to \mrm{R}\opn{Hom}_A(\mrm{R}\opn{Hom}_A(M,R),R)
\]
is an isomorphism in $\cat{D}(A)$.

Similarly to the (non-)uniqueness theorem for dualizing complexes over rings,
a similar result is true over DG-rings.
In particular, if $(A,\bar{\m})$ is a noetherian local DG-ring,
and if $R_1,R_2$ are dualizing DG-modules over $A$,
then there exists $n \in \mathbb{Z}$ such that $R_1 \cong R_2[n]$.
We say that a dualizing DG-module $R$ over a noetherian local DG-ring $(A,\bar{\m})$ is normalized if 
$\inf(R) = -\dim(\mrm{H}^0(A))$.

\subsection{Local cohomology over commutative DG-rings}\label{sec:lc}

Following \cite{BIK,Sh1}, let us recall the notion of local cohomology over commutative DG-rings.
Let $A$ be a commutative DG-ring,
and let $\bar{\a} \subseteq \mrm{H}^0(A)$ be a finitely generated ideal.
Recall that an $\mrm{H}^0(A)$-module $\bar{M}$ is called $\bar{\a}$-torsion if
for any $\bar{m} \in \bar{M}$ there exists $n \in \mathbb{N}$ such that $\bar{\a}^n \cdot \bar{m} = 0$,
equivalently, if 
\[
\bar{M} = \varinjlim \opn{Hom}_{\mrm{H}^0(A)}(\mrm{H}^0(A)/\bar{\a}^n,\bar{M}).
\]
The category of all $\bar{\a}$-torsion modules is a thick abelian subcategory of $\opn{Mod}(\mrm{H}^0(A))$.
This implies that the 
category $\cat{D}_{\bar{\a}-\opn{tor}}(A)$ consisting of DG-modules $M$ such that for all $n \in \mathbb{Z}$,
the $\mrm{H}^0(A)$-module $\mrm{H}^n(M)$ is $\bar{\a}$-torsion, is a triangulated subcategory of $\cat{D}(A)$.
One can show (see \cite{BIK,Sh1} for details) that the inclusion functor 
\[
\cat{D}_{\bar{\a}-\opn{tor}}(A) \inj \cat{D}(A)
\]
has a right adjoint
\[
\cat{D}(A) \to \cat{D}_{\bar{\a}-\opn{tor}}(A),
\]
and composing this right adjoint with the inclusion,
one obtains a triangulated functor
\[
\mrm{R}\Gamma_{\bar{\a}}:\cat{D}(A) \to \cat{D}(A),
\]
which we call the derived torsion or local cohomology functor of $A$ with respect to $\bar{\a}$.

In case $A = \mrm{H}^0(A)$ is a commutative noetherian ring,
this coincide with Grothendieck's local cohomology functor,
the (total) right derived functor of the $\bar{\a}$-torsion functor
\[
\Gamma_{\bar{\a}}(-) := \varinjlim \opn{Hom}_{A}(A/\bar{\a}^n,-) : \opn{Mod}(A) \to \opn{Mod}(A).
\]

\textbf{Warning:} If $A = \mrm{H}^0(A)$ is a ring which is not noetherian,
the functor $\mrm{R}\Gamma_{\bar{\a}}$ might be different in general from 
the right derived functor of the $\bar{\a}$-torsion functor.
In this paper, when we write $\mrm{R}\Gamma_{\bar{\a}}$ we will always mean the functor introduced above, 
which may be different from the right derived functor of the $\bar{\a}$-torsion functor
(The condition which guarantees these functors coincide is that the ideal $\bar{\a}$ is weakly proregular,
which is always the case if $\mrm{H}^0(A)$ is noetherian).

\begin{rem}\label{rem:lc-notation}
Given a commutative DG-ring $A$ and a finitely generated ideal $\bar{\a} \subseteq \mrm{H}^0(A)$,
we will sometimes need to work both local cohomology functor of $A$ with respect to $\bar{\a}$,
and with the local cohomology functor of $\mrm{H}^0(A)$ with respect to $\bar{\a}$.
The former is a functor $\cat{D}(A) \to \cat{D}(A)$,
while the latter is a functor $\cat{D}(\mrm{H}^0(A)) \to \cat{D}(\mrm{H}^0(A))$.
The notation we introduced above doesn't allow one to distinguish between the two.
To fix this issue, when we work with both of these functors,
we will denote the former by
\[
\mrm{R}\Gamma_{\bar{\a}}^A : \cat{D}(A) \to \cat{D}(A),
\]
and the latter by
\[
\mrm{R}\Gamma_{\bar{\a}}^{\bar{A}} : \cat{D}(\mrm{H}^0(A)) \to \cat{D}(\mrm{H}^0(A)).
\]
\end{rem}

To actually compute the functor $\mrm{R}\Gamma_{\bar{\a}}(-)$,
we recall the telescope complex,
following \cite{GM,PSY}.
Given a commutative ring $A$,
and given $a \in A$,
the telescope complex associated to $A$ and $a$ is the complex
\[
0 \to \bigoplus_{n=0}^{\infty} A \to \bigoplus_{n=0}^{\infty} A \to 0
\]
in degrees $0,1$, with the differential being defined by
\[
d(e_i) =      \left\{ \begin{array}{lr} e_0, & \text{if $i = 0$,}\\
																				e_{i-1}-a\cdot e_i, & \text{if $i \ge 1$}.
											\end{array}\right.
\]
where we have let $e_0,e_1,\dots$ denote the standard basis of the countably generated free $A$-module $\oplus_{n=0}^{\infty} A$.
We denote this complex by $\opn{Tel}(A;a)$.
Given a finite sequence $\mathbf{\a} = a_1,\dots, a_n \in A$,
we set
\[
\opn{Tel}(A;a_1,\dots,a_n) = \opn{Tel}(A;a_1) \otimes_A \opn{Tel}(A;a_2) \otimes_A \dots \otimes_A \opn{Tel}(A;a_n).
\]
The complex $\opn{Tel}(A;\mathbf{\a})$ is a bounded complex of free $A$-modules,
called the telescope complex associated to $A$ and $\mathbf{\a}$.

One important property of the telescope complex is its behavior with respect to base change.
Let $A$ be a commutative ring, 
let $\mathbf{\a}$ be a finite sequence of elements of $A$,
let $B$ be another commutative ring,
and let $f:A \to B$ be a ring homomorphism.
Denoting by $\mathbf{\b}$ the image of $\mathbf{\a}$ by $f$,
there is an isomorphism of complexes of $B$-modules:
\[
\opn{Tel}(A;\mathbf{a}) \otimes_A B \cong \opn{Tel}(B;\mathbf{b}).
\]

As is well known (see for instance \cite[Proposition 4.8]{PSY}),
if $A$ is a noetherian ring,
$\a \subseteq A$,
and $\mathbf{\a}$ is a finite sequence of elements of $A$ that generates $\a$,
there is a natural isomorphism
\[
\mrm{R}\Gamma_{\a}(M) \cong \opn{Tel}(A;\mathbf{\a}) \otimes_A M
\]
for every $M \in \cat{D}(A)$.

More generally, as shown in \cite[Corollary 2.13]{Sh1},
if $A$ is a commutative DG-ring and $\bar{\a} \subseteq \mrm{H}^0(A)$ is a finitely generated ideal,
and if $\mathbf{\a} = (a_1,\dots,a_n)$ is a finite sequence of elements of $A^0$,
whose image in $\mrm{H}^0(A)$ generates $\bar{\a}$,
then there is a natural isomorphism
\begin{equation}\label{eqn:RGamma}
\mrm{R}\Gamma_{\bar{\a}}(M) \cong \opn{Tel}(A^0;\mathbf{\a})\otimes_{A^0} A\otimes_A M
\end{equation}
for any $M \in \cat{D}(A)$.

If $\bar{\a},\bar{\b} \subseteq \mrm{H}^0(A)$ are two finitely generated ideals such that
$\sqrt{\bar{\a}} = \sqrt{\bar{\b}}$,
then by \cite[Corollary 2.15]{Sh1},
there is a natural isomorphism
\begin{equation}\label{eqn:RGammaRadical}
\mrm{R}\Gamma_{\bar{\a}}(-) \cong \mrm{R}\Gamma_{\bar{\b}}(-).
\end{equation}

\subsection{Derived completion of DG-modules and derived completion of DG-rings}\label{sec:DerComp}

Let $A$ be a commutative DG-ring,
and let $\bar{\a} \subseteq \mrm{H}^0(A)$ be a finitely generated ideal.
As explained in \cite{Sh1},
the local cohomology functor $\mrm{R}\Gamma_{\bar{\a}}:\cat{D}(A) \to \cat{D}(A)$
has a left adjoint, 
which we denote by $\mrm{L}\Lambda_{\bar{\a}}:\cat{D}(A) \to \cat{D}(A)$,
and call the derived $\bar{\a}$-adic completion functor.
The reason for this name is that if $A = \mrm{H}^0(A)$ is a noetherian ring,
it coincides with the left derived functor of the $\bar{\a}$-adic completion functor
$\Lambda_{\bar{\a}}(-) := \varprojlim -\otimes_A A/{\bar{\a}}^n$
which was introduced in \cite{GM}. 
As for derived torsion, in case $A$ is a ring, the condition under which these two operations coincide is that $\bar{\a}$ is a weakly proregular ideal. In case $A$ is a ring, the left adjoint we discuss here is studied in \cite[Section 091N]{SP}.

Similarly to the previous section,
if $A$ is a commutative DG-ring and $\bar{\a} \subseteq \mrm{H}^0(A)$ is a finitely generated ideal,
and if $\mathbf{\a} = (a_1,\dots,a_n)$ is a finite sequence of elements of $A^0$,
whose image in $\mrm{H}^0(A)$ generates $\bar{\a}$,
then there is a natural isomorphism
\[
\mrm{L}\Lambda_{\bar{\a}}(M) \cong \opn{Hom}_{A^0}(\opn{Tel}(A^0;\mathbf{\a})\otimes_{A^0} A,M)
\]
for any $M \in \cat{D}(A)$.

Given a commutative DG-ring $A$ and a finitely generated ideal $\bar{\a} \subseteq \mrm{H}^0(A)$,
it is explained \cite[Section 4]{Sh1} that the DG-module $\mrm{L}\Lambda_{\bar{\a}}(A)$
has the structure of a commutative non-positive DG-ring,
called the derived $\bar{\a}$-adic completion of $A$.
We denote this DG-ring by $\mrm{L}\Lambda(A,\bar{\a})$.
If $A$ is noetherian then its derived completion is also noetherian,
and if $(A,\bar{\m})$ is a noetherian local DG-ring,
then $\mrm{L}\Lambda(A,\bar{\m})$ is a noetherian local DG-ring.
Furthermore, we also have the following result, complementing \cite{Sh1}:

\begin{prop}\label{prop:amp-of-dc}
Let $(A,\bar{\m})$ be a commutative noetherian local DG-ring,
and let $B = \mrm{L}\Lambda(A,\bar{\m})$ be the derived $\bar{\m}$-adic completion of $A$.
Then for all $i \le 0$, we have that
\[
\mrm{H}^i(B) \cong \Lambda_{\bar{\m}}(\mrm{H}^i(A)).
\]
In particular $\amp(B) = \amp(A)$.
\end{prop}
\begin{proof}
Let $E:=E(A,\bar{\m})$ be the injective DG-module corresponding to the maximal ideal $\bar{\m}$.
By \cite[Theorem 7.22]{Sh2}, we have that 
\[
B \cong \mrm{R}\opn{Hom}_A(E,E).
\]
Hence, it follows from \cite[Theorem 4.10, Corollary 4.12]{Sh2} that
\begin{align}
\mrm{H}^i(B) = \mrm{H}^i\left(\mrm{R}\opn{Hom}_A(E,E)\right) \cong
\opn{Hom}_{\mrm{H}^0(A)}(\mrm{H}^{-i}(E),\mrm{H}^0(E)) \cong\nonumber\\
\opn{Hom}_{\mrm{H}^0(A)}\left(\opn{Hom}_{\mrm{H}^0(A)}(\mrm{H}^i(A),\mrm{H}^0(E)),\mrm{H}^0(E)\right).\nonumber
\end{align}

Since $\mrm{H}^0(E)$ is exactly the injective hull of the residue field of the local ring $\mrm{H}^0(A)$,
the result follows from Matlis duality over the noetherian local ring $(\mrm{H}^0(A),\bar{\m})$.
Finally, the equality $\amp(B) = \amp(A)$ follows from faithfulness of adic completion on finitely generated modules.
\end{proof}

\subsection{Localization and support}

Given a commutative DG-ring $A$, 
we recall, following \cite[Section 4]{Ye1}, 
that one may localize it at at prime ideals of $\mrm{H}^0(A)$.
Given a prime $\bar{\p} \in \opn{Spec}(\mrm{H}^0(A))$,
the localization $A_{\bar{\p}}$ is defined as follows:
let $\pi_A:A \to \mrm{H}^0(A)$ be the canonical surjection,
and $\pi_A^0:A^0 \to \mrm{H}^0(A)$ its degree $0$ component.
Let $\p = (\pi_A^0)^{-1}(\bar{\p})$. 
Then $\p \in \opn{Spec}(A^0)$,
and one sets
\[
A_{\bar{\p}} := A \otimes_{A^0} A^0_{\p}.
\]
More generally, given $M \in \cat{D}(A)$,
we define
\[
M_{\bar{\p}} := M\otimes_A A_{\bar{\p}} \cong M\otimes_{A^0} A^0_{\p} \in \cat{D}(A_{\bar{\p}}).
\]
Since $A^0_{\p}$ is flat over $A_0$,
it follows that for all $n \in \mathbb{Z}$,
we have that

\begin{equation}\label{eqn:loc-coho}
\mrm{H}^n(A_{\bar{\p}}) = \mrm{H}^n(A)_{\bar{\p}}, \quad  \mrm{H}^n(M_{\bar{\p}}) = \mrm{H}^n(M)_{\bar{\p}}.
\end{equation}

It follows that if $A$ is a noetherian DG-ring,
then $(A_{\bar{\p}},\bar{\p}\cdot \mrm{H}^0(A)_{\bar{\p}})$ is a noetherian local DG-ring,
and that $\amp(A_{\bar{\p}}) \le \amp(A)$.

\begin{lem}\label{lem:local-ten-hz}
Let $A$ be a commutative DG-ring,
and let $M \in \cat{D}(A)$.
For any $\bar{\p} \in \opn{Spec}(\mrm{H}^0(A))$,
there is a natural isomorphism
\[
M_{\bar{\p}} \otimes^{\mrm{L}}_A \mrm{H}^0(A) \cong M\otimes^{\mrm{L}}_A \mrm{H}^0(A)_{\bar{\p}}
\]
in $\cat{D}(\mrm{H}^0(A))$.
\end{lem}
\begin{proof}
This follows from associativity of the derived tensor product,
and the fact that $A_{\bar{\p}} \otimes^{\mrm{L}}_A \mrm{H}^0(A) \cong \mrm{H}^0(A)_{\bar{\p}}$.
\end{proof}

\begin{dfn}\label{dfn:support}
Let $A$ be a commutative DG-ring,
and let $M \in \cat{D}(A)$.
We define the support of $M$ over $A$ to be the set
\[
\opn{Supp}_A(M) := \{\bar{\p} \in \opn{Spec}(\mrm{H}^0(A)) \mid M_{\bar{\p}} \ncong 0\}.
\]
\end{dfn}

It follows from the definition and from (\ref{eqn:loc-coho}) that
\[
\opn{Supp}_A(M) = \bigcup_{n \in \mathbb{Z}} \opn{Supp}_{\mrm{H}^0(A)}\left(\mrm{H}^n(M)\right).
\]

\begin{prop}\label{prop:supp-red}
Let $A$ be a commutative DG-ring,
and let $M \in \cat{D}^{-}(A)$.
Then there is an equality
\[
\opn{Supp}_A(M) = \opn{Supp}_{\mrm{H}^0(A)}\left(M\otimes^{\mrm{L}}_A \mrm{H}^0(A)\right).
\]
\end{prop}
\begin{proof}
Given $\bar{\p} \in \opn{Spec}(\mrm{H}^0(A))$,
we have that $\bar{\p} \in \opn{Supp}_A(M)$ 
if and only if $M_{\bar{\p}} \ncong 0$.
Note that $M_{\bar{\p}} \in \cat{D}^{-}(A)$.
By (\ref{eqn:red-ten}), we have that $M_{\bar{\p}} \ncong 0$ if and only if
$M_{\bar{\p}}\otimes^{\mrm{L}}_A \mrm{H}^0(A) \ncong 0$.
By Lemma \ref{lem:local-ten-hz}, this is equivalent to
\[
M\otimes^{\mrm{L}}_A \mrm{H}^0(A)_{\bar{\p}} \cong \left(M\otimes^{\mrm{L}}_A \mrm{H}^0(A)\right)_{\bar{\p}} \ncong 0
\]
if and only if
\[
\bar{\p} \in \opn{Supp}_{\mrm{H}^0(A)}\left(M\otimes^{\mrm{L}}_A \mrm{H}^0(A)\right),
\]
as claimed.
\end{proof}

\section{Local cohomology Krull dimension over commutative local DG-rings}\label{sec:KDim}

The next definition follows \cite[Section 3]{Fo}:

\begin{dfn}\label{dfn:krulldim}
Let $(A,\bar{\m})$ be a noetherian local DG-ring,
and let $M \in \cat{D}^{-}(A)$ be a bounded above DG-module.
We define the \textbf{local cohomology Krull dimension} of $M$ to be
\[
\lcdim(M) := \sup_{\ell \in \mathbb{Z}} \left\{\dim(\mrm{H}^{\ell}(M)) + \ell\right\}
\]
where $\dim(\mrm{H}^{\ell}(M))$ is the usual Krull dimension of the $\mrm{H}^0(A)$-module 
$\mrm{H}^{\ell}(M)$; that is, the Krull dimension of the noetherian ring $\mrm{H}^0(A)/\opn{ann}(\mrm{H}^{\ell}(M))$.
\end{dfn}

\begin{rem}
The name \textit{local cohomology Krull dimension} is justified by Theorem \ref{thm:non-vanish} below.
\end{rem}

\begin{rem}
Since any $\mrm{H}^0(A)$-module $\bar{M}$ has 
\[
0 \le \dim(\bar{M}) \le \dim(\mrm{H}^0(A)),
\]
we necessarily have 
\begin{equation}\label{eqn:dim}
\sup(M) \le \lcdim(M) \le \sup(M) + \dim(\mrm{H}^0(A))
\end{equation}
for any $M \in \cat{D}^{-}(A)$,
with $\lcdim(M) = \sup(M) + \dim(\mrm{H}^0(A))$ if and only if 
\begin{equation}\label{eqn:max-dim}
\dim(\mrm{H}^{\sup(M)}(M)) = \dim(\mrm{H}^0(A)).
\end{equation}
In particular, it follows, since $A$ is non-positive, that 
\begin{equation}\label{eqn:dimofA}
\lcdim(A) = \dim(\mrm{H}^0(A)).
\end{equation}
\end{rem}

\begin{rem}
The paper \cite{BSW} discusses another notion of Krull dimension for differential graded algebras, 
and shows that in some nice cases, 
it coincides with $\dim(\mrm{H}^0(A))$, as in the above definition.
\end{rem}

We now discuss two results we will need about bounds of local cohomology over commutative local rings 
(and not DG-rings, as in the rest of this paper).
These results are probably well known, 
and we simply wish to emphasize that they also hold for unbounded complexes,
as we will need to apply them in an unbounded situation.

\begin{prop}\label{prop:dim-in-rings}
Let $(A,\m)$ be a noetherian local ring,
and let $M \in \cat{D}^{-}_{\mrm{f}}(A)$.
Then
\[
\lcdim(M) = \sup\left( \mrm{R}\Gamma_{\m}(M) \right).
\]
\end{prop}
\begin{proof}
If $M$ has bounded cohomology, then this is precisely \cite[Proposition 3.14(d)]{Fo}.
In the general case, let $d = \dim(A)$. Then $\amp(\mrm{R}\Gamma_{\m}(A)) \le d$.
Let $n = \sup(M) - (d+1)$, and set
\[
M' := \opn{smt}^{\le n}(M), \quad M'' := \opn{smt}^{> n}(M).
\]
Then there is a distinguished triangle
\[
M' \to M \to M'' \to M'[1]
\]
in $\cat{D}(A)$. Applying the triangulated functor $\mrm{R}\Gamma_{\m}(-)$, 
we obtain a distinguished triangle:
\[
\mrm{R}\Gamma_{\m}(M') \to \mrm{R}\Gamma_{\m}(M) \to \mrm{R}\Gamma_{\m}(M'') \to \mrm{R}\Gamma_{\m}(M')[1]
\]
If $i < \sup(M) -d$, we must have 
\[
i+\dim(\mrm{H}^i(M)) < \sup(M)
\]
so it follows that $\lcdim(M) = \lcdim(M'')$.
Since $\mrm{R}\Gamma_{\m}(A) \in [0,d]$, and 
$\mrm{R}\Gamma_{\m}(M') \cong \mrm{R}\Gamma_{\m}(A) \otimes^{\mrm{L}}_A M'$,
we have that 
\[
\mrm{H}^i\left(\mrm{R}\Gamma_{\m}(M')\right) = 0 \quad, \mbox{for $i\ge \sup(M)$}
\]
Hence, the above distinguished triangle and the fact that $M''$ is a bounded complex imply there are equalities
\[
\sup\left(\mrm{R}\Gamma_{\m}(M)\right) = \sup\left(\mrm{R}\Gamma_{\m}(M'')\right) = \lcdim(M'') = \lcdim(M).
\]
\end{proof}
Similarly, the following unbounded version of \cite[Proposition 3.7]{Fo} holds.
We omit the similar proof.

\begin{prop}\label{prop:classical-bound}
Let $(A,\m)$ be a noetherian local ring,
and let $M \in \cat{D}^{-}(A)$.
Then 
\[
\sup\left(\mrm{R}\Gamma_{\m}(M)\right) \le \lcdim(M).
\]
\end{prop}

\begin{prop}\label{prop:RGamma-and-Tensor}
Let $A,B$ be commutative DG-rings,
let $\bar{\a} \subseteq \mrm{H}^0(A)$ be a finitely generated ideal,
let $f:A \to B$ be a map of DG-rings,
and let 
\[
\bar{\b} := \mrm{H}^0(f)(\bar{\a})\cdot \mrm{H}^0(B)
\]
be the ideal in $\mrm{H}^0(B)$ generated by the image of $\bar{\a}$.
Given $M \in \cat{D}(A)$,
there is a natural isomorphism
\[
\mrm{R}\Gamma_{\bar{\a}}(M) \otimes^{\mrm{L}}_A B \cong \mrm{R}\Gamma_{\bar{\b}}(M \otimes^{\mrm{L}}_A B)
\]
in $\cat{D}(B)$.
\end{prop}
\begin{proof}
Let $\bar{x}_1,\dots,\bar{x}_n$ be a finite sequence of elements in $\mrm{H}^0(A)$ that generates $\bar{\a}$,
and let $x_1,\dots,x_n$ be lifts of these elements to $A^0$.
Then by \cite[Corollary 2.13]{Sh1}, we have that
\[
\mrm{R}\Gamma_{\bar{\a}}(M) \cong \opn{Tel}(A^0;x_1,\dots,x_n)\otimes_{A^0} A \otimes_A M.
\]
Hence, by the base change property of the telescope complex 
\[
\mrm{R}\Gamma_{\bar{\a}}(M) \otimes^{\mrm{L}}_A B \cong
\opn{Tel}(B^0;f(x_1),\dots,f(x_n))\otimes_{B^0} \left(M \otimes^{\mrm{L}}_A B\right).
\]
Since the images of $f(x_1),\dots,f(x_n)$ in $\mrm{H}^0(B)$ generate $\bar{\b}$,
we deduce from \cite[Corollary 2.13]{Sh1} that the latter is naturally isomorphic to 
$\mrm{R}\Gamma_{\bar{\b}}(M \otimes^{\mrm{L}}_A B)$,
as claimed.
\end{proof}

In the next results we will use the notation $\mrm{R}\Gamma^{\bar{A}}$ and $\mrm{R}\Gamma^A$ introduced in Remark \ref{rem:lc-notation}.
Similarly, we will discuss both the local cohomology Krull dimension of DG-modules over a DG-ring $A$ and of complexes over the ring $\mrm{H}^0(A)$.
To distinguish between them, we will denote the former by $\lcdim_A(-)$ and the latter by $\lcdim_{\bar{A}}(-)$.

\begin{prop}\label{prop:rgammatohz}
Let $A$ be a noetherian DG-ring,
let $\bar{\a} \subseteq \mrm{H}^0(A)$ be a an ideal,
and let $M \in \cat{D}(A)$.
Then there is a natural isomorphism
\[
\mrm{R}\Gamma^A_{\bar{\a}}(M) \otimes^{\mrm{L}}_A \mrm{H}^0(A)
\cong 
\mrm{R}\Gamma^{\bar{A}}_{\bar{\a}}\left(M \otimes^{\mrm{L}}_A \mrm{H}^0(A)\right)
\]
in $\cat{D}(\mrm{H}^0(A))$.
\end{prop}
\begin{proof}
This follows from applying Proposition \ref{prop:RGamma-and-Tensor} to the map $A\to \mrm{H}^0(A)$.
The assumption that $\mrm{H}^0(A)$ is noetherian is required,
in order for the local cohomology functor discussed in the DG context to coincide with the classical local cohomology functor 
$\mrm{R}\Gamma^{\bar{A}}_{\bar{\a}}$.
\end{proof}

\begin{lem}\label{lem:dim-reduction}
Let $(A,\bar{\m})$ be a noetherian local DG-ring,
and let $M \in \cat{D}^{-}(A)$.
Then there is an equality
\[
\sup\left(\mrm{R}\Gamma^A_{\bar{\m}}(M)\right) = \sup\left(\mrm{R}\Gamma^{\bar{A}}_{\bar{\m}}(M \otimes^{\mrm{L}}_A \mrm{H}^0(A)) \right).
\]
If moreover $M \in \cat{D}^{-}_{\mrm{f}}(A)$, then there is also an equality
\[
\sup\left(\mrm{R}\Gamma^A_{\bar{\m}}(M)\right) = \lcdim_{\bar{A}} ( M\otimes^{\mrm{L}}_A \mrm{H}^0(A) ).
\]
\end{lem}
\begin{proof}
Letting $x_1,\dots,x_n \in A^0$ be a finite sequence of elements of $A^0$ such that their image in $\mrm{H}^0(A)$ generates $\bar{\m}$,
it follows from (\ref{eqn:RGamma}) that
\[
\mrm{R}\Gamma_{\bar{\m}}(M) \cong \opn{Tel}(A^0;x_1,\dots,x_n) \otimes_{A^0} A \otimes_A M
\]
Since the DG-module $\opn{Tel}(A^0;x_1,\dots,x_n) \otimes_{A^0} A$ is bounded-above, we deduce that the DG-module
$\mrm{R}\Gamma_{\bar{\m}}(M)$ is also bounded above.
Hence, by (\ref{eqn:red-ten}) we have that
\[
\sup\left(\mrm{R}\Gamma^A_{\bar{\m}}(M)\right) = \sup\left(\mrm{R}\Gamma^A_{\bar{\m}}(M) \otimes^{\mrm{L}}_A \mrm{H}^0(A) \right) =
\sup\left(\mrm{R}\Gamma^{\bar{A}}_{\bar{\m}}(M \otimes^{\mrm{L}}_A \mrm{H}^0(A)) \right)
\]
where the second equality follows from Proposition \ref{prop:rgammatohz}.
This proves the first claim.
Finally, if in addition $M \in \cat{D}^{-}_{\mrm{f}}(A)$, then we have that
\[
M \otimes^{\mrm{L}}_A \mrm{H}^0(A) \in \cat{D}^{-}_{\mrm{f}}(\mrm{H}^0(A)),
\]
so it follows from Proposition \ref{prop:dim-in-rings} that
\[
\sup\left(\mrm{R}\Gamma^{\bar{A}}_{\bar{\m}}(M \otimes^{\mrm{L}}_A \mrm{H}^0(A)) \right)
= \lcdim_{\bar{A}} ( M\otimes^{\mrm{L}}_A \mrm{H}^0(A) ).
\]
proving the second claim.
\end{proof}

\begin{prop}\label{prop:dim-reduction}
Let $(A,\bar{\m})$ be a noetherian local DG-ring,
and let $M \in \cat{D}^{-}(A)$.
Then there is an equality
\[
\lcdim_A(M) = \lcdim_{\bar{A}}(M\otimes^{\mrm{L}}_A \mrm{H}^0(A)).
\]
\end{prop}
\begin{proof}
By our definition,
\[
\lcdim_A(M) = \sup_{\ell \in \mathbb{Z}} \left\{\dim(\mrm{H}^{\ell}(M)) + \ell\right\}.
\]
First, exactly as in the proof of \cite[Proposition 3.5]{Fo}, note that this number satisfies 
\begin{equation}\label{eqn:dim-alt-equ}
\lcdim(M) = \sup_{\bar{\p} \in \opn{Spec}(\mrm{H}^0(A))} \left\{\dim(\mrm{H}^0(A)/\bar{\p}) + \sup(M_{\bar{\p}}) \right\}.
\end{equation}
Now, given $\bar{\p} \in \opn{Spec}(\mrm{H}^0(A))$,
by (\ref{eqn:red-ten}) we have that
\[
\sup(M_{\bar{\p}}) = \sup\left(M_{\bar{\p}} \otimes^{\mrm{L}}_A \mrm{H}^0(A)\right).
\]
By Lemma \ref{lem:local-ten-hz}, we have that
\[
M_{\bar{\p}} \otimes^{\mrm{L}}_A \mrm{H}^0(A) \cong 
\left(M\otimes^{\mrm{L}}_A \mrm{H}^0(A)\right)_{\bar{\p}}.
\]
Hence, using \cite[Proposition 3.5]{Fo} we see that $\lcdim_A(M)$ is equal to
\[
\sup_{\bar{\p} \in \opn{Spec}(\mrm{H}^0(A)} \left\{\dim(\mrm{H}^0(A)/\bar{\p}) + \sup\left(M\otimes^{\mrm{L}}_A \mrm{H}^0(A)\right)_{\bar{\p}} \right\} = \lcdim_{\bar{A}}(M\otimes^{\mrm{L}}_A \mrm{H}^0(A))
\]
as claimed.
\end{proof}

The next result is a DG version of Grothendieck's vanishing and non-vanishing theorems for local cohomology
(see \cite[Theorems 6.1.2 and 6.1.4]{BS}).

\begin{thm}\label{thm:non-vanish}
Let $(A,\bar{\m})$ be a noetherian local DG-ring,
and $M \in \cat{D}^{-}(A)$. 
Then
\[
\sup\left(\mrm{R}\Gamma_{\bar{\m}}(M)\right) \le \lcdim(M).
\]
If moreover $M \in \cat{D}^{-}_{\mrm{f}}(A)$ then
\[
\sup\left(\mrm{R}\Gamma_{\bar{\m}}(M)\right) = \lcdim(M).
\]
In particular, for $M \in \cat{D}^{-}_{\mrm{f}}(A)$ we have that
\[
\mrm{H}^{\lcdim(M)}_{\bar{\m}}(M) :=
\mrm{H}^{\lcdim(M)}\left(\mrm{R}\Gamma_{\bar{\m}}(M)\right) \ne 0.
\]
\end{thm}
\begin{proof}
Take $M \in \cat{D}^{-}(A)$.
By Lemma \ref{lem:dim-reduction}
we have that
\[
\sup\left(\mrm{R}\Gamma^A_{\bar{\m}}(M)\right) = \sup\left(\mrm{R}\Gamma^{\bar{A}}_{\bar{\m}}(M \otimes^{\mrm{L}}_A \mrm{H}^0(A)) \right),
\]
and by Proposition \ref{prop:classical-bound}, 
since $M \otimes^{\mrm{L}}_A \mrm{H}^0(A) \in \cat{D}^{-}(\mrm{H}^0(A))$, we have that
\[
\sup\left(\mrm{R}\Gamma^{\bar{A}}_{\bar{\m}}(M \otimes^{\mrm{L}}_A \mrm{H}^0(A)) \right) \le \lcdim_{\bar{A}}(M \otimes^{\mrm{L}}_A \mrm{H}^0(A))
\]
so the first claim follows from Proposition \ref{prop:dim-reduction}.
Now, assume further that $M \in \cat{D}^{-}_{\mrm{f}}(A)$.
Then by Lemma \ref{lem:dim-reduction}
\[
\sup\left(\mrm{R}\Gamma^A_{\bar{\m}}(M)\right) = \lcdim_{\bar{A}} ( M\otimes^{\mrm{L}}_A \mrm{H}^0(A) ),
\]
and by Proposition \ref{prop:dim-reduction} the latter is equal to $\lcdim(M)$, as claimed.
\end{proof}

\begin{cor}\label{cor:non-vanish}
Let $(A,\bar{\m})$ be a noetherian local DG-ring,
and let $d = \dim(\mrm{H}^0(A))$.
Then
\[
\mrm{H}^{d}\left(\mrm{R}\Gamma_{\bar{\m}}(A)\right) \ne 0.
\]
and 
\[
\mrm{H}^{i}\left(\mrm{R}\Gamma_{\bar{\m}}(A)\right) = 0
\]
for all $i>d$.
\end{cor}
\begin{proof}
This follows from (\ref{eqn:dimofA}) and Theorem \ref{thm:non-vanish}.
\end{proof}

\begin{prop}\label{prop:dim-of-dc}
Let $(A,\bar{\m})$ be a commutative noetherian local DG-ring,
and let $B = \mrm{L}\Lambda(A,\bar{\m})$ be the derived $\bar{\m}$-adic completion of $A$.
Then $\lcdim(B) = \lcdim(A)$.
\end{prop}
\begin{proof}
According to \cite[Proposition 4.16]{Sh1},
we have that 
\[
\mrm{H}^0(B) = \Lambda_{\bar{\m}}(\mrm{H}^0(A)).
\]
Hence, since $B$ is non-positive, we obtain
\[
\lcdim(B) = \dim\left(\mrm{H}^0(B)\right) = \dim\left(\Lambda_{\bar{\m}}(\mrm{H}^0(A))\right) = \dim\left(\mrm{H}^0(A)\right) = \lcdim(A).
\]
\end{proof}

\section{Depth and local cohomology over commutative local DG-rings}\label{sec:depth}

Our definition of depth is identical to the usual homological definition of depth over local rings:

\begin{dfn}
Let $(A,\bar{\m},\k)$ be a noetherian local DG-ring,
and let $M \in \cat{D}^{+}(A)$.
We define the depth of $M$ to be the number
\[
\depth_A(M) := \inf\left(\mrm{R}\opn{Hom}_A(\k,M)\right).
\]
\end{dfn}

Note that this definition is not invariant under translations,
but is very useful for working with local cohomology. 
In Definition \ref{dfn:seq} below we will give a modified definition of depth that is invariant under translations.
It follows from the definition that $\depth_A(0) = -\infty$.

Dually to Proposition \ref{prop:dim-reduction}, 
we have the following reduction formula for the depth, 
and, as in the case of rings, a connection to local cohomology:

\begin{prop}\label{prop:depth}
Let $(A,\bar{\m},\k)$ be a noetherian local DG-ring.
Then for any $M \in \cat{D}^{+}(A)$, there are equalities
\[
\inf\left(\mrm{R}\Gamma_{\bar{\m}}(M)\right) = \depth_A(M) = \depth_{\mrm{H}^0(A)} \left(\mrm{R}\opn{Hom}_A(\mrm{H}^0(A),M) \right).
\]
\end{prop}
\begin{proof}
Since $M \in \cat{D}^{+}(A)$, and since the functor $\mrm{R}\Gamma_{\bar{\m}}(-)$ has finite cohomological dimension, $\mrm{R}\Gamma_{\bar{\m}}(M) \in \cat{D}^{+}(A)$.
Hence, by (\ref{eqn:red-rhom}), 
\[
\inf \left( \mrm{R}\Gamma_{\bar{\m}} (M) \right) = \inf \left( \mrm{R}\opn{Hom}_A(\mrm{H}^0(A), \mrm{R}\Gamma_{\bar{\m}}(M)) \right).
\]
According to \cite[Proposition 7.23]{Sh2}, there is an isomorphism
\[
\mrm{R}\opn{Hom}_A(\mrm{H}^0(A),\mrm{R}\Gamma^A_{\bar{\m}}(M)) \cong \mrm{R}\Gamma^{\bar{A}}_{\bar{\m}}\left(\mrm{R}\opn{Hom}_A(\mrm{H}^0(A),M)\right),
\]
in $\cat{D}(\mrm{H}^0(A))$, so it is enough to compute the infimum of the latter.
Since 
\[
\mrm{R}\opn{Hom}_A(\mrm{H}^0(A),M) \in \cat{D}^{+}(\mrm{H}^0(A)),
\]
it follows from \cite[Proposition 3.8]{Fo} that 
\[
\inf \left(  \mrm{R}\Gamma^{\bar{A}}_{\bar{\m}}\left(\mrm{R}\opn{Hom}_A(\mrm{H}^0(A),M)\right) \right) = \inf\left( \mrm{R}\opn{Hom}_{\mrm{H}^0(A)}(\k,\mrm{R}\opn{Hom}_A(\mrm{H}^0(A),M))  \right).
\]
The latter is by definition 
\[
\depth_{\mrm{H}^0(A)} \left(\mrm{R}\opn{Hom}_A(\mrm{H}^0(A),M) \right).
\]
Furthermore, the adjunction isomorphism shows that
\[
\mrm{R}\opn{Hom}_{\mrm{H}^0(A)}(\k,\mrm{R}\opn{Hom}_A(\mrm{H}^0(A),M)) \cong \mrm{R}\opn{Hom}_A(\k,M).
\]
proving the claim.
\end{proof}

Over a local ring, it follows immediately from the definition of local cohomology that the depth of any complex is greater or equal to its infimum, and, as explained, in \cite[(3.3)]{Fo}, 
they are equal if and only if the maximal ideal is an associated prime of the bottommost cohomology of the complex.
Similarly, we have:

\begin{prop}\label{prop:depth-lower-bound}
Let $(A,\bar{\m},\k)$ be a noetherian local DG-ring.
Then for any $0 \ne M \in \cat{D}^{+}(A)$, there are inequalities
\[
\inf\left(\mrm{R}\Gamma_{\bar{\m}}(M)\right) = \depth_A(M) \ge \inf(M).
\]
Moreover, there is an equality 
\[
\depth_A(M) = \inf(M)
\]
if and only if $\bar{\m}$ is an associated prime of $\mrm{H}^{\inf(M)}(M)$.
\end{prop}
\begin{proof}
By Proposition \ref{prop:depth},
\[
\depth_A(M) = \depth_{\mrm{H}^0(A)} \left(\mrm{R}\opn{Hom}_A(\mrm{H}^0(A),M) \right)
\]
As remarked above, 
\[
\depth_{\mrm{H}^0(A)} \left(\mrm{R}\opn{Hom}_A(\mrm{H}^0(A),M) \right) \ge \inf \left(\mrm{R}\opn{Hom}_A(\mrm{H}^0(A),M) \right),
\]
with equality if and only if $\bar{\m}$ is an associated prime of 
\[
\mrm{H}^{\inf \left(\mrm{R}\opn{Hom}_A(\mrm{H}^0(A),M) \right)}\left(\mrm{R}\opn{Hom}_A(\mrm{H}^0(A),M)\right).
\]
Hence, the result follows from the equalities
\[
\inf \left(\mrm{R}\opn{Hom}_A(\mrm{H}^0(A),M) \right) = \inf(M)
\]
and
\[
\mrm{H}^{\inf(M)}\left(\mrm{R}\opn{Hom}_A(\mrm{H}^0(A),M) \right) = \mrm{H}^{\inf(M)}(M)
\]
of (\ref{eqn:red-rhom}).
\end{proof}

We shall need the following upper bound satisfied by depth over rings:

\begin{prop}\label{prop:depth-ring-bound}
Let $(A,\m)$ be a noetherian local ring,
and let $M \in \cat{D}^{+}_{\mrm{f}}(A)$.
Then there is an inequality:
\[
\depth_A(M) \le \dim(\mrm{H}^{\inf(M)}(M)) + \inf(M).
\]
\end{prop}
\begin{proof}
If $M \in \cat{D}^{\mrm{b}}_{\mrm{f}}(A)$,
then this statement is exactly \cite[Proposition 3.17]{Fo}.
In the general case,
let $n := \dim(\mrm{H}^{\inf(M)}(M)) + \inf(M)$,
and consider the following distinguished triangle in $\cat{D}(A)$:
\[
M' \to M \to M'' \to M'[1]
\]
where
\[
M' := \opn{smt}^{\le n}(M), \quad M'' := \opn{smt}^{> n}(M).
\]
Applying the triangulated functor $\mrm{R}\Gamma_{\m}$,
and passing to cohomology, we have for each $i \in \mathbb{Z}$ the following exact sequence of local cohomology modules:
\[
\mrm{H}_{\m}^{i-1}(M'') \to \mrm{H}_{\m}^{i}(M') \to \mrm{H}_{\m}^{i}(M) \to \mrm{H}_{\m}^{i}(M'')
\]
From the definition of local cohomology, and since $\inf(M'') > n$,
it follows that 
\[
\mrm{H}_{\m}^{i}(M'') = 0
\]
for all $i \le n$.
Hence, 
\[
\mrm{H}_{\m}^{i}(M') \cong \mrm{H}_{\m}^{i}(M)
\]
for all $i \le n$.
Since $M' \in \mrm{D}^{\mrm{b}}_{\mrm{f}}(A)$,
according to \cite[Proposition 3.17]{Fo} there exists 
$i \le \dim(\mrm{H}^{\inf(M')}(M')) + \inf(M')$ such that $\mrm{H}_{\m}^{i}(M') \ne 0$.
Since 
\[
\dim(\mrm{H}^{\inf(M')}(M')) + \inf(M') =  \dim(\mrm{H}^{\inf(M)}(M)) + \inf(M) = n,
\]
we deduce that there exists some
\[
i \le \dim(\mrm{H}^{\inf(M)}(M)) + \inf(M)
\]
such that $\mrm{H}_{\m}^{i}(M) \ne 0$, as claimed.
\end{proof}

The next result is a DG-version of Proposition \ref{prop:depth-ring-bound}.
\begin{prop}\label{prop:depth-dg-bound}
Let $(A,\bar{\m})$ be a noetherian local DG-ring,
and let $M \in \cat{D}^{+}_{\mrm{f}}(A)$.
Then there is an inequality:
\[
\depth_A(M) \le \dim(\mrm{H}^{\inf(M)}(M)) + \inf(M).
\]
\end{prop}
\begin{proof}
By Proposition \ref{prop:depth},
we have that
\[
\depth_A(M) = \depth_{\mrm{H}^0(A)}\left(\mrm{R}\opn{Hom}_A(\mrm{H}^0(A),M)\right).
\]
Since $M \in \cat{D}^{+}_{\mrm{f}}(A)$, 
we have that $\mrm{R}\opn{Hom}_A(\mrm{H}^0(A),M) \in \cat{D}^{+}_{\mrm{f}}(\mrm{H}^0(A))$.
Hence, by Proposition \ref{prop:depth-ring-bound},
we obtain an inequality
\[
\depth_{\mrm{H}^0(A)}\left(\mrm{R}\opn{Hom}_A(\mrm{H}^0(A),M)\right) \le \dim\left(\mrm{H}^i(\mrm{R}\opn{Hom}_A(\mrm{H}^0(A),M) + i \right)
\]
where $i := \inf\left(\mrm{R}\opn{Hom}_A(\mrm{H}^0(A),M)\right)$.
According to (\ref{eqn:red-rhom}), there are equalities
\[
\inf(M) = \inf\left(\mrm{R}\opn{Hom}_A(\mrm{H}^0(A),M)\right), \quad
\mrm{H}^i(M) = \mrm{H}^i\left(\mrm{R}\opn{Hom}_A(\mrm{H}^0(A),M)\right).
\]
Combining all of the above, we see that
\[
\depth_A(M) = \depth_{\mrm{H}^0(A)}\left(\mrm{R}\opn{Hom}_A(\mrm{H}^0(A),M)\right)
\le dim(\mrm{H}^{\inf(M)}(M)) + \inf(M),
\]
as claimed.
\end{proof}

\begin{cor}\label{cor:rgammainf}
Let $(A,\bar{\m})$ be a noetherian local DG-ring with bounded cohomology,
and let $d = \dim(\mrm{H}^0(A))$.
Then there is an inequality
\[
\inf\left(\mrm{R}\Gamma_{\bar{\m}}(A)\right) \le \inf(A) + d.
\]
If there is an equality $\inf\left(\mrm{R}\Gamma_{\bar{\m}}(A)\right) = \inf(A) + d$
then
\begin{equation}
\dim\left(\mrm{H}^{\inf(A)}(A)\right) = \dim(\mrm{H}^0(A)).
\end{equation}
\end{cor}
\begin{proof}
By Proposition \ref{prop:depth} and Proposition \ref{prop:depth-dg-bound}
we have that
\begin{equation}\label{eqn:exact-bound}
\inf\left(\mrm{R}\Gamma_{\bar{\m}}(A)\right) = \depth_A(A) \le \dim(\mrm{H}^{\inf(A)}(A)) + \inf(A).
\end{equation}
Since $\mrm{H}^{\inf(A)}(A)$ is an $\mrm{H}^0(A)$-module,
we have that $\dim(\mrm{H}^{\inf(A)}(A) \le d$, proving the claim.
If $\inf\left(\mrm{R}\Gamma_{\bar{\m}}(A)\right) = \inf(A) + d$,
it follow from the (\ref{eqn:exact-bound}) that $d \le \dim(\mrm{H}^{\inf(A)}(A))$,
which implies that $\dim(\mrm{H}^{\inf(A)}(A)) = d$.
\end{proof}

\begin{prop}\label{prop:depth-of-dc}
Let $(A,\bar{\m})$ be a noetherian local DG-ring with bounded cohomology,
and let $B = \mrm{L}\Lambda(A,\bar{\m})$ be the derived $\bar{\m}$-adic completion of $A$.
Then $\depth(B) = \depth(A)$.
\end{prop}
\begin{proof}
In this proof we will use the terminology of \cite{Sh1}.
Let $(C,\mathbf{c})$ be a weakly proregular resolution of $(A,\bar{\m})$ 
(in the sense of \cite[Definition 2.1]{Sh1}).
Since $C \to A$ is a quasi-isomorphism,
we have that $\depth(C) = \depth(A)$.
Let $\c$ be the ideal in $C^0$ generated by $\mathbf{c}$,
and let $\bar{\c}:= \c\cdot \mrm{H}^0(A)$.
The isomorphism $\mrm{H}^0(C) \to \mrm{H}^0(A)$ sends $\bar{\c}$ to $\bar{\m}$.
According to \cite[Theorem 4.8]{Sh1}, there is an isomorphism
\begin{equation}\label{eqn:compute-llambda}
B = \mrm{L}\Lambda(A,\bar{\m}) \cong \Lambda_{\c}(C) = \varprojlim_n C\otimes_{C^0} C^0/\c^n.
\end{equation}
Let us denote the latter by $\widehat{C}$.
The ideal of definition of the local DG-ring $\widehat{C}$ is given by
\[
\widehat{\bar{\c}} := \bar{\c} \cdot \mrm{H}^0(\widehat{C}).
\]
The isomorphism (\ref{eqn:compute-llambda}) implies that $\depth(B) = \depth(\widehat{C})$.
By Proposition \ref{prop:RGamma-and-Tensor}, 
there is an isomorphism
\[
\mrm{R}\Gamma_{\bar{\c}}(C) \otimes^{\mrm{L}}_C \widehat{C} \cong \mrm{R}\Gamma_{\widehat{\bar{\c}}}(\widehat{C}).
\]
Hence, we have that
\[
\depth(\widehat{C}) = \inf\left(\mrm{R}\Gamma_{\bar{\c}}(C) \otimes^{\mrm{L}}_C \widehat{C}\right).
\]
To compute the infimum of the latter,
we may apply the forgetful functor $\cat{D}(\widehat{C}) \to \cat{D}(C)$,
and consider the DG-module
\[
\mrm{R}\Gamma_{\bar{\c}}(C) \otimes^{\mrm{L}}_C \widehat{C} \in \cat{D}(C).
\]
It follows from \cite[Corollary 2.13]{Sh1} that there is an isomorphism
\[
\widehat{C} \cong \mrm{L}\Lambda_{\bar{\c}}(C)
\]
in $\cat{D}(C)$, so there is an isomorphism
\[
\mrm{R}\Gamma_{\bar{\c}}(C) \otimes^{\mrm{L}}_C \widehat{C} \cong
\mrm{R}\Gamma_{\bar{\c}}\left(\mrm{L}\Lambda_{\bar{\c}}(C)\right).
\]
By the MGM equivalence (a result dual to \cite[Proposition 2.7]{Sh1} proven similarly to it),
we have that
\[
\mrm{R}\Gamma_{\bar{\c}}\left(\mrm{L}\Lambda_{\bar{\c}}(C)\right) \cong \mrm{R}\Gamma_{\bar{\c}}(C),
\]
so that
\[
\inf\left(\mrm{R}\Gamma_{\bar{\c}}(C) \otimes^{\mrm{L}}_C \widehat{C}\right) = \depth(C),
\]
which implies that
\[
\depth(A) = \depth(C) = \depth(\widehat{C}) = \depth(B),
\]
as claimed.
\end{proof}

\section{Local-Cohen-Macaulay commutative DG-rings}\label{sec:CMDG}

We are now ready to prove items (1) and (3) of Theorem \ref{thm:main-bounds} from the introduction.
They are contained in the following result:

\begin{thm}\label{thm:main-amp}
The following inequalities hold:
\begin{enumerate}[wide, labelwidth=!, labelindent=0pt]
\item If $(A,\bar{\m})$ is a noetherian local DG-ring with $\amp(A) < \infty$ and $d = \dim(\mrm{H}^0(A))$ then
\[
\amp(A) \le \amp\left(\mrm{R}\Gamma_{\bar{\m}}(A)\right) \le \amp(A) + d.
\]
\item If $A$ is a noetherian DG-ring, 
and $R$ is a dualizing DG-module over $A$ then 
\[
\amp(A) \le \amp(R).
\]
If moreover $A$ has is local and has bounded cohomology, and $d = \dim(\mrm{H}^0(A))$ then $d < \infty$ and
\[
\amp(R) \le \amp(A) + d.
\]
\end{enumerate}
\end{thm}
\begin{proof}
\begin{enumerate}[wide, labelwidth=!, labelindent=0pt]
\item Let $(A,\bar{\m})$ be a noetherian local DG-ring.
Let $n = \amp(A)$, $d = \dim(\mrm{H}^0(A))$, and suppose that $n < \infty$.
By Corollary \ref{cor:non-vanish}, we have that
\[
\sup\left(\mrm{R}\Gamma_{\bar{\m}}(A)\right) = d.
\]
By Corollary \ref{cor:rgammainf}, there is an inequality
\[
\inf\left(\mrm{R}\Gamma_{\bar{\m}}(A)\right) \le \inf(A) + d = d-n.
\]
Combining these two facts we obtain:
\[
\amp\left(\mrm{R}\Gamma_{\bar{\m}}(A)\right) = \sup\left(\mrm{R}\Gamma_{\bar{\m}}(A)\right) - \inf\left(\mrm{R}\Gamma_{\bar{\m}}(A)\right) \ge d+(n-d) = n = \amp(A).
\]
On the other hand, by Proposition \ref{prop:depth-lower-bound},
\[
\inf\left(\mrm{R}\Gamma_{\bar{\m}}(A)\right) \ge \inf(A) = -n,
\]
so that
\[
\amp\left(\mrm{R}\Gamma_{\bar{\m}}(A)\right) \le d + n = \amp(A) + d.
\]
\item Let $A$ be a noetherian DG-ring,
and $R$ a dualizing DG-module over $A$.
If $\amp(A) = \infty$ then by \cite[Corollary 7.3]{Ye1} we have that $\amp(R) = \infty$.
We may thus assume that $\amp(A) = n < \infty$.
Let $\bar{\p} \in \opn{Supp}(\mrm{H}^{-n}(A))$.
Then $A_{\bar{\p}}$ is a noetherian local DG-ring,
and since $\mrm{H}^{-n}(A)_{\bar{\p}} \ne 0$,
there is an equality $\amp(A) = \amp(A_{\bar{\p}})$.
Since localization of DG-rings is cohomologically essentially smooth (in the sense of \cite[Definition 6.4]{Sh}),
according to \cite[Corollary 6.11]{Sh} the DG-module $R_{\bar{\p}}$ is a dualizing DG-module over $A_{\bar{\p}}$.
By local duality for local DG-rings (\cite[Corollary 7.29]{Sh2}),
there is an equality
\[
\amp\left(\mrm{R}\Gamma_{\bar{\p}}(A_{\bar{\p}})\right) = \amp\left(R_{\bar{\p}}\right).
\]
By (1) of this theorem, we have that 
\[
\amp\left(\mrm{R}\Gamma_{\bar{\p}}(A_{\bar{\p}})\right) \ge \amp(A_{\bar{\p}}) = \amp(A)
\]
and since $\amp(R) \ge \amp\left(R_{\bar{\p}}\right)$, 
we deduce that $\amp(R) \ge \amp(A)$.
Next, observe that by \cite[Proposition 7.5]{Ye1}, 
the complex $\mrm{R}\opn{Hom}_A(\mrm{H}^0(A),R)$ is a dualizing complex over $\mrm{H}^0(A)$,
so by \cite[Corollary V.7.2]{RD} we have that $d = \dim(\mrm{H}^0(A)) < \infty$.
Finally, assuming that $A$ is local, the inequality
\[
\amp(R) \le \amp(A) + d.
\]
follows from local duality for local DG-rings and from (1).
\end{enumerate}
\end{proof}

In view of Theorem \ref{thm:main-amp}, it makes sense to define:
\begin{dfn}\label{def:CM}
Let $(A,\bar{\m})$ be a noetherian local DG-ring with $\amp(A) < \infty$.
We say that $A$ is local-Cohen-Macaulay if $\amp(\mrm{R}\Gamma_{\bar{\m}}(A)) = \amp(A)$.
\end{dfn}

\begin{exa}
Let $(A,\m)$ be a noetherian local ring. 
Then $A$ is Cohen-Macaulay in the classical sense if and only if $A$ is local-Cohen-Macaulay in the sense of Definition \ref{def:CM}.
\end{exa}

\begin{prop}\label{prop:CM-by-DC}
Let $(A,\bar{\m})$ be a noetherian local DG-ring with $\amp(A) < \infty$,
and let $R$ be a dualizing DG-module over $A$.
Then $A$ is local-Cohen-Macaulay if and only if $\amp(A) = \amp(R)$.
\end{prop}
\begin{proof}
This follows immediately from the equality $\amp(\mrm{R}\Gamma_{\bar{\m}}(A)) = \amp(R)$ established in \cite[Corollary 7.29]{Sh2}.
\end{proof}

Following \cite{FJ,FIJ}, 
recall that a noetherian local DG-ring $(A,\bar{\m})$ is called Gorenstein 
if $\amp(A) < \infty$ and $\opn{inj}\dim_A(A) < \infty$.
In this case, note that $A$ is a dualizing DG-module over $A$.
Hence, by Proposition \ref{prop:CM-by-DC} we have:

\begin{prop}\label{prop:GorisCM}
Let $(A,\bar{\m})$ be a noetherian local Gorenstein DG-ring.
Then $A$ is local-Cohen-Macaulay.
\end{prop}
\qed

Just like noetherian local rings,
noetherian local DG-rings need not to have dualizing DG-modules.
However, passing to their derived completion, 
we showed in \cite[Proposition 7.21]{Sh2} that the derived completion has a dualizing DG-module.
It is thus convenient to know that the local-Cohen-Macaulay property is preserved by the derived completion operation.
In Example \ref{exa:no-dualizing} below we construct a local-Cohen-Macaulay DG-ring $A$ which is not equivalent to a ring,
such that $A$ does not have a dualizing DG-module.

\begin{prop}
Let $(A,\bar{\m})$ be a noetherian local DG-ring with $\amp(A) < \infty$,
and let $B = \mrm{L}\Lambda(A,\bar{\m})$ be its derived $\bar{\m}$-adic completion.
Then $A$ is local-Cohen-Macaulay if and only if $B$ is local-Cohen-Macaulay.
\end{prop}
\begin{proof}
This follows from the equalities
\[
\amp(A) = \amp(B), \quad \lcdim(A) = \lcdim(B), \quad \depth(A) = \depth(B)
\]
shown in Propositions \ref{prop:amp-of-dc}, \ref{prop:dim-of-dc} and \ref{prop:depth-of-dc}.
\end{proof}

Next we wish to show that zero-dimensional DG-rings are local-Cohen-Macaulay. 
First we need the following lemma about local cohomology with respect to nilpotent ideals over DG-rings:

\begin{lem}\label{lem:lc-of-nil}
Let $A$ be a noetherian DG-ring,
let $\bar{\a} \subseteq \mrm{H}^0(A)$ be an ideal,
and suppose that there is some $n \in \mathbb{N}$ such that $\bar{\a}^n = 0$.
Then the functor $\mrm{R}\Gamma_{\bar{\a}}:\cat{D}(A) \to \cat{D}(A)$ is isomorphic to the identity functor.
\end{lem}
\begin{proof}
Let $M \in \cat{D}(A)$.
In general (even without the nilpotent assumption),
it follows from \cite[Corollary 2.13]{Sh1} that
\[
\mrm{R}\Gamma_{\bar{\a}}(M) \cong \mrm{R}\Gamma_{\bar{\a}}(A) \otimes^{\mrm{L}}_A M,
\]
so it is enough to show that $\mrm{R}\Gamma_{\bar{\a}}(A) \cong A$.
By Proposition \ref{prop:rgammatohz}, we have that
 \[
\mrm{R}\Gamma^A_{\bar{\a}}(A) \otimes^{\mrm{L}}_A \mrm{H}^0(A)
\cong 
\mrm{R}\Gamma^{\bar{A}}_{\bar{\a}}\left(\mrm{H}^0(A)\right) \cong \mrm{H}^0(A)
\]
where the last isomorphism follows from the fact that since $\bar{\a}$ is nilpotent,
the additive functor 
\[
\Gamma_{\bar{\a}}:\opn{Mod}\left(\mrm{H}^0(A)\right) \to \opn{Mod}\left(\mrm{H}^0(A)\right)
\]
is equal to the identity functor, 
so its right derived functor is also the identity functor.
Hence, by \cite[Proposition 3.3(1)]{Ye1}, 
it follows that $\mrm{R}\Gamma_{\bar{\a}}(A) \cong A$,
as claimed.
\end{proof}

\begin{prop}\label{prop:zd-is-cm}
Let $(A,\bar{\m})$ be a noetherian local DG-ring with $\amp(A) < \infty$,
and suppose that $\dim(\mrm{H}^0(A)) = 0$.
Then $A$ is local-Cohen-Macaulay.
\end{prop}
\begin{proof}
Since $(\mrm{H}^0(A),\bar{\m})$ is a zero-dimensional local ring,
its maximal ideal $\bar{\m}$ must be nilpotent.
Hence, by Lemma \ref{lem:lc-of-nil} we have
\[
\amp\left(\mrm{R}\Gamma_{\bar{\m}}(A)\right) = \amp(A),
\]
so that $A$ is local-Cohen-Macaulay.
\end{proof}

\begin{exa}\label{exa:finite}
Let $\varphi:(A,\m) \to (B,\n)$ be a local homomorphism between noetherian local rings.
Assume that $\varphi$ is a finite ring map, 
and that it makes $B$ an $A$-module of finite flat dimension.
Let $D = A/\m \otimes^{\mrm{L}}_A B$ be the derived fiber of $\varphi$.
This is a commutative local DG-ring with $\mrm{H}^0(D) = B/\m B$.
The assumption that $\varphi$ is finite implies that $B/\m B$ is a zero dimensional ring,
and the finite flat dimension assumption implies that $\amp(D) < \infty$.
Hence, by Proposition \ref{prop:zd-is-cm}, $D$ is a local-Cohen-Macaulay DG-ring.
\end{exa}

Given a local homomorphism $(A,\m) \to (B,\n)$ of finite flat dimension,
it follows from \cite[Theorem 4.4]{AF1} and \cite[Theorem 7.1]{AF1} that
if $A$ is Gorenstein then $B$ is Gorenstein if and only if the local DG-ring 
$D = A/\m \otimes^{\mrm{L}}_A B$ is Gorenstein. 
Unfortunately, it turns out not to be the case for the local-Cohen-Macaulay property,
as the next example shows:

\begin{exa}
Let $(A,\m)$ be a regular local ring,
and let $B = A/I$ be a quotient of $A$ which is not Cohen-Macaulay.
For a specific example, on can take $A = \k[[x,y,z]]$,
$B = \k[[x,y,z]]/(xy,xz)$.
Since $A$ is regular, $B$ has finite flat dimension over $A$,
so by Example \ref{exa:finite},
the derived fiber $D$ of the map $\varphi: A \to B$ is a local-Cohen-Macaulay DG-ring.
Since $B$ is not a Cohen-Macaulay ring, 
it follows from \cite[(8.9)]{AF2} that the map $\varphi$ is not a Cohen-Macaulay homomorphism in the sense of \cite[(8.1)]{AF2}.
Note also that since $B$ is not Gorenstein,
it follows from \cite[Theorem 7.1]{AF1} that $D$ is not a Gorenstein DG-ring.
\end{exa}

The minimal non-vanishing cohomology of a local-Cohen-Macaulay DG-ring satisfies the following property:

\begin{prop}\label{prop:dim-of-inf-cm}
Let $(A,\bar{\m})$ be a noetherian local-Cohen-Macaulay DG-ring,
and let $n = \amp(A)$.
Then the $\mrm{H}^0(A)$-module $\mrm{H}^{-n}(A)$ satisfies
\[
\dim\left(\mrm{H}^{-n}(A)\right) = \dim\left(\mrm{H}^0(A)\right).
\]
\end{prop}
\begin{proof}
Let $d = \dim(\mrm{H}^0(A))$.
By the proof of Theorem \ref{thm:main-amp},
if $A$ is local-Cohen-Macaulay, we must have that
\[
\inf\left(\mrm{R}\Gamma_{\bar{\m}}(A)\right) = \inf(A) + d.
\]
Hence, by equation (\ref{eqn:exact-bound}) in Corollary \ref{cor:rgammainf},
we deduce that
\[
\dim\left(\mrm{H}^{-n}(A)\right) = \dim\left(\mrm{H}^0(A)\right).
\] 
\end{proof}

\section{Regular sequences and the derived Bass conjecture}\label{sec:regseq}

In this section we study regular sequences and system of parameters over noetherian local DG-rings.
Much of this section is inspired, and based on,
the work of Christensen on regular sequences and system of parameters acting on chain complexes over local rings (\cite{ChI,ChII}).
We first recall the notions of a regular sequence in the DG-setting, 
and the notion of a quotient DG-module, essentially following Minamoto (\cite[Section 3.2]{Mi2}).

Given a commutative DG-ring $A$ and $\bar{x} \in \mrm{H}^0(A)$,
the identification 
\[
\mrm{H}^0(A) = \mrm{H}^0\left(\mrm{R}\opn{Hom}_A(A,A)\right) = \opn{Hom}_{\cat{D}(A)}(A,A)
\]
implies that $\bar{x}$ induces a map $A \to A$ in $\cat{D}(A)$,
which we also denote by $\bar{x}$.
We denote the mapping cone of the map $\bar{x}$ by $A//\bar{x}$,
so there is a distinguished triangle

\begin{equation}\label{eqn:xmap-tri}
A \xrightarrow{\bar{x}} A \to A//\bar{x} \to A[1]
\end{equation}
in $\cat{D}(A)$. 
It is shown in \cite[Section 3.2]{Mi2} that $A//\bar{x}$ has the structure of a commutative DG-ring.
If $A$ is noetherian, then $A//\bar{x}$ is also noetherian, 
and if $\amp(A) < \infty$ then $\amp(A//\bar{x}) < \infty$.
It follows from (\ref{eqn:xmap-tri}) and the fact that $A$ is non-positive that
\[
\mrm{H}^0(A//\bar{x}) \cong \mrm{H}^0(A)/\bar{x}.
\]
In particular, if $A$ is a noetherian local DG-ring,
then $A//\bar{x}$ is also a noetherian local DG-ring.

Given $M \in \cat{D}(A)$,
applying the triangulated functor $M \otimes^{\mrm{L}}_A -$ to the triangle (\ref{eqn:xmap-tri}),
we obtain another distinguished triangle in $\cat{D}(A)$:
\[
M \xrightarrow{\bar{x}\otimes^{\mrm{L}}_A M} M \to M//\bar{x}M \to M[1]
\]
where we have set 
\[
M//\bar{x}M := A//\bar{x}\otimes^{\mrm{L}}_A M.
\]
In particular we see that $M//\bar{x}M \in \cat{D}(A//\bar{x})$.

Given a finite set of elements $\bar{x}_1,\dots,\bar{x}_n \in \mrm{H}^0(A)$,
setting $B = A//\bar{x}_1$,
we define inductively 
\[
A//(\bar{x}_1,\dots,\bar{x}_n) = B//(\bar{x}_2,\dots,\bar{x}_n),
\]
where we identified $\bar{x}_2,\dots,\bar{x}_n$ with their images in $\mrm{H}^0(A)/\bar{x}_1$.
Similarly one defines 
\[
M//(\bar{x}_1,\dots,\bar{x}_n)M = M\otimes^{\mrm{L}}_A A//(\bar{x}_1,\dots,\bar{x}_n) \in \cat{D}(A//(\bar{x}_1,\dots,\bar{x}_n)).
\]

\begin{dfn}\label{dfn:seq}
Let $A$ be a commutative DG-ring,
and let $M \in \cat{D}^{+}(A)$.
\begin{enumerate}
\item An element $\bar{x} \in \mrm{H}^0(A)$ is called $M$-regular if it is $\mrm{H}^{\inf(M)}(M)$-regular;
that is, if the multiplication map
\[
\bar{x} \times - : \mrm{H}^{\inf(M)}(M) \to \mrm{H}^{\inf(M)}(M)
\]
is injective.
\item Inductively, a sequence $\bar{x}_1,\dots \bar{x}_n \in \mrm{H}^0(A)$ is called $M$-regular if
$\bar{x}_1$ is $M$-regular, and the sequence $\bar{x}_2, \dots \bar{x}_n$ is $M//\bar{x}_1M$-regular.
\item Assuming $(A,\bar{\m})$ is a noetherian local DG-ring,
the \textbf{sequential depth} of $M$,
denoted by $\opn{seq.depth}_A(M)$ is the length of an $M$-regular sequence contained in $\bar{\m}$ of maximal length.
\end{enumerate}
\end{dfn}

It is shown in \cite[Proposition 3.15]{Mi2} that if $(A,\bar{\m})$ is a noetherian local DG-ring,
and if $0 \ne M \in \cat{D}^{+}_{\mrm{f}}(A)$ then
\begin{equation}\label{eqn:chdepth}
\opn{seq.depth}_A(M) = \depth_A(M) - \inf(M).
\end{equation}
To be precise, Minamoto calls the right hand side of this equality the cohomological depth of $M$,
and shows in that proposition that it coincides with what we called here the sequential depth of $M$.
In view of this formula, 
we make the following definition which is a variation on the local cohomology Krull dimension of a DG-module which is not affected by shifts.

\begin{dfn}
Given a noetherian local DG-ring $(A,\bar{\m})$, and given $M \in \cat{D}^{-}(A)$,
we define the derived Krull dimension of $M$ to be the number
\[
\opn{der.dim}_A(M) = \lcdim_A(M) - \sup(M).
\]
\end{dfn}

\begin{cor}\label{cor:chdepth}
Let $(A,\bar{\m})$ be a noetherian local DG-ring with $\amp(A) < \infty$.
Then there is an inequality
\[
\opn{seq.depth}_A(A) \le \dim(\mrm{H}^0(A)) = \opn{der.dim}_A(A),
\]
with equality if and only if $A$ is local-Cohen-Macaulay.
\end{cor}
\begin{proof}
It follows from Theorem \ref{thm:non-vanish} and Proposition \ref{prop:depth} that
\[
\amp\left(\mrm{R}\Gamma_{\bar{\m}}(A)\right) = \lcdim(A) - \depth_A(A).
\]
By Theorem \ref{thm:main-amp}(1),
since $\sup(A) = 0$, we have that
\[
\lcdim(A) - \depth_A(A) \ge \amp(A) = -\inf(A)
\]
which implies by (\ref{eqn:chdepth}) that
\[
\opn{seq.depth}_A(A) \le \lcdim(A) = \opn{der.dim}_A(A).
\]
Moreover, we see that this is an equality if and only if
\[
\amp\left(\mrm{R}\Gamma_{\bar{\m}}(A)\right) = \lcdim(A) - \depth_A(A) = \amp(A)
\]
if and only if $A$ is local-Cohen-Macaulay.
\end{proof}

Assuming $\bar{x} \in \bar{\m}$,
it is shown in \cite[Lemma 3.13]{Mi2} that 
\[
\inf(M)-1 \le \inf(M//\bar{x}M) \le \inf(M),
\]
and that $\inf(M//\bar{x}M) = \inf(M)$ if and only if $\bar{x}$ is $M$-regular.
We further note that by Nakayama's lemma,
it is always the case that $\sup(M//\bar{x}M) = \sup(M)$.
It follows that if $M \in \cat{D}^{\mrm{b}}_{\mrm{f}}(A)$,
and if $\bar{x} \in \bar{\m}$ is $M$-regular,
then 
\begin{equation}\label{eqn:reg-has-same-amp}
\amp(M//\bar{x}M) = \amp(M).
\end{equation}

\begin{rem}
Minamoto's discussion of the above in \cite{Mi2} uses a slightly different terminology,
as instead of working with elements of $\mrm{H}^0(A)$,
he works with lifting of them to the ring $A^0$.
Given $M \in \cat{D}(A)$ and $n \in \mathbb{Z}$,
since the $A^0$ action on $\mrm{H}^n(M)$ factors through $\mrm{H}^0(A)$,
it follows that our definitions are equivalent to his.
\end{rem}

The next definitions are DG-versions of \cite[Definition 2.3]{ChI} and \cite[Definition 2.1]{ChII}.
We note that \cite{ChI,ChII} uses homological notation, 
making the definitions look slightly different.
The support of a DG-module was defined in Definition \ref{dfn:support}.

\begin{dfn}
Let $(A,\bar{\m})$ be a commutative noetherian local DG-ring.
\begin{enumerate}
\item Given $M \in \cat{D}^{+}(A)$,
the set of associated primes of $M$ is given by
\[
\opn{Ass}_A(M) := \{\bar{\p} \in \opn{Supp}_A(M) \mid \depth_{A_{\bar{\p}}}(M_{\bar{\p}}) = \inf(M_{\bar{\p}})\}.
\]
\item Given $M \in \cat{D}^{-}(A)$,
we set
\[
W_0^A(M) := \{\bar{\p} \in \opn{Supp}_A(M) \mid \lcdim_A(M) \le \sup(M_{\bar{\p}}) + \dim(\mrm{H}^0(A)/\bar{\p})\}.
\]
\end{enumerate}
\end{dfn}

We now show that these sets of prime ideals are often finite.

\begin{prop}\label{prop:assfinite}
Let $(A,\bar{\m})$ be a commutative noetherian local DG-ring,
and let $M \in \cat{D}^{\mrm{b}}_{\mrm{f}}(A)$.
Then the set $\opn{Ass}_A(M)$ of associated primes of $M$ is a finite set.
\end{prop}
\begin{proof}
Let $\bar{\p} \in \opn{Ass}_A(M)$,
so that $\depth_{A_{\bar{\p}}}(M_{\bar{\p}}) = \inf(M_{\bar{\p}})$.
It follows from Proposition \ref{prop:depth-lower-bound} that
$\bar{\p}\cdot\mrm{H}^0(A)_{\bar{\p}}$ is an associated prime of the $\mrm{H}^0(A)_{\bar{\p}}$-module
\[
\mrm{H}^{\inf(M_{\bar{\p}})}(M_{\bar{\p}}) \cong \left(\mrm{H}^{\inf(M_{\bar{\p}})}(M)\right)_{\bar{\p}}.
\]
This implies (for instance, by \cite[Tag 0310]{SP})
that $\bar{\p}$ is an associated prime of the $\mrm{H}^0(A)$-module
$\mrm{H}^{\inf(M_{\bar{\p}})}(M)$.
It follows that
\[
\opn{Ass}_A(M) \subseteq \bigcup_{n=\inf(M)}^{\sup(M)} \opn{Ass}_{\mrm{H}^0(A)}\left(\mrm{H}^n(M)\right).
\]
Since $\mrm{H}^n(M)$ is a finitely generated $\mrm{H}^0(A)$-module,
it has only finitely many associated primes,
so the fact that $\mrm{H}^n(M) \ne 0$ only for finitely many $n \in \mathbb{Z}$ implies the result.
\end{proof}

\begin{lem}\label{lem:dim-inequal-local}
Let $(A,\bar{\m})$ be a commutative noetherian local DG-ring,
let $M \in \cat{D}^{\mrm{b}}_{\mrm{f}}(A)$,
and let $\bar{\p} \in \opn{Supp}_A(M)$.
Then there is an inequality
\[
\lcdim_A(M) \ge \lcdim_{A_{\bar{\p}}}(M_{\bar{\p}}) + \dim(\mrm{H}^0(A)/\bar{\p}).
\]
\end{lem}
\begin{proof}
By the definition of the local cohomology Krull dimension,
there is some $n \in \mathbb{Z}$ such that $\mrm{H}^n(M_{\bar{\p}}) \ne 0$ and
\[
\lcdim_{A_{\bar{\p}}}(M_{\bar{\p}}) = n + \dim_{\bar{A}_{\bar{\p}}}\left(\mrm{H}^n(M_{\bar{\p}})\right).
\]
By basic properties of Krull dimension, 
the finitely generated $\mrm{H}^0(A)$-module $\mrm{H}^n(M)$ satisfies 
\[
\dim\left(\mrm{H}^n(M)_{\bar{\p}}\right) + \dim\left(\mrm{H}^0(A)/\bar{\p}\right) \le \dim\left(\mrm{H}^n(M)\right).
\]
Hence, we obtain
\begin{align}
\lcdim_{A_{\bar{\p}}}(M_{\bar{\p}}) + \dim(\mrm{H}^0(A)/\bar{\p}) = \nonumber\\
n + \dim_{\bar{A}_{\bar{\p}}}\left(\mrm{H}^n(M_{\bar{\p}})\right) + \dim(\mrm{H}^0(A)/\bar{\p}) 
\le\nonumber\\
n + \dim\left(\mrm{H}^n(M)\right) \le \lcdim_A(M) \nonumber
\end{align}
as claimed.
\end{proof}

\begin{prop}\label{prop:W0finite}
Let $(A,\bar{\m})$ be a commutative noetherian local DG-ring,
and let $M \in \cat{D}^{\mrm{b}}_{\mrm{f}}(A)$.
Then the set $W_0^A(M)$ is a finite set.
\end{prop}
\begin{proof}
Let $\bar{\p} \in W_0^A(M)$.
By the definition of $W_0^A(M)$ and Lemma \ref{lem:dim-inequal-local}
there are inequalities
\[
\lcdim_{A_{\bar{\p}}}(M_{\bar{\p}}) + \dim(\mrm{H}^0(A)/\bar{\p}) \le \lcdim_A(M) \le \sup(M_{\bar{\p}}) + \dim(\mrm{H}^0(A)/\bar{\p}).
\]
This implies that 
\[
\lcdim_{A_{\bar{\p}}}(M_{\bar{\p}}) \le \sup(M_{\bar{\p}}).
\]
so by (\ref{eqn:dim}) there is an equality
\begin{equation}\label{eqn:dim-of-loc-is-sup}
\lcdim_{A_{\bar{\p}}}(M_{\bar{\p}}) = \sup(M_{\bar{\p}}).
\end{equation}
Since by definition of the local cohomology Krull dimension, 
we have that
\[
\lcdim_{A_{\bar{\p}}}(M_{\bar{\p}}) \ge \sup(M_{\bar{\p}}) + \dim_{\bar{A}_{\bar{\p}}}\left(  \mrm{H}^{\sup(M_{\bar{\p}})}(M_{\bar{\p}}) \right),
\]
it follows from (\ref{eqn:dim-of-loc-is-sup}) that
\[
\dim_{\bar{A}_{\bar{\p}}}\left(  \mrm{H}^{\sup(M_{\bar{\p}})}(M_{\bar{\p}}) \right) = 0.
\]
Since there is an isomorphism
\[
\mrm{H}^{\sup(M_{\bar{\p}})}(M_{\bar{\p}}) \cong \left( \mrm{H}^{\sup(M_{\bar{\p}})}(M) \right)_{\bar{\p}}, 
\]
we deduce that $\bar{\p}$ is a minimal prime ideal of the finitely generated $\mrm{H}^0(A)$-module 
\[
\mrm{H}^{\sup(M_{\bar{\p}})}(M).
\]
It follows that
\[
W_0^A(M) \subseteq \bigcup_{n=\inf(M)}^{\sup(M)} \opn{Min}_{\mrm{H}^0(A)}\left(\mrm{H}^n(M)\right).
\]
Since $\mrm{H}^n(M)$ is a finitely generated $\mrm{H}^0(A)$-module,
it has only finitely many minimal primes,
so the fact that $\mrm{H}^n(M) \ne 0$ only for finitely many $n \in \mathbb{Z}$ implies the result.
\end{proof}

\begin{prop}\label{prop:out-of-ass-reg}
Let $(A,\bar{\m})$ be a noetherian local DG-ring,
let $M \in \cat{D}^{+}(A)$,
let $\bar{x} \in \bar{m} \subseteq \mrm{H}^0(A)$,
and assume that
\[
\bar{x} \notin \bigcup_{\bar{\p} \in \opn{Ass}_A(M)} \bar{\p}.
\]
Then $\bar{x}$ is $M$-regular.
\end{prop}
\begin{proof}
Suppose $\bar{x}$ is not $M$-regular.
Then the map
\[
\bar{x}:\mrm{H}^{\inf(M)}(M) \to \mrm{H}^{\inf(M)}(M)
\]
is not injective, so that $\bar{x} \in \mrm{H}^0(A)$ is a zero-divisor for $\mrm{H}^{\inf(M)}(M)$.
Hence, by \cite[Tag 00LD]{SP}, 
there exists an associated prime $\bar{\p}$ of the $\mrm{H}^0(A)$-module $\mrm{H}^{\inf(M)}(M)$ such that 
$\bar{x} \in \bar{\p}$. 
By \cite[Tag 0310]{SP}, we deduce that $\bar{\p}\cdot\mrm{H}^0(A)_{\bar{\p}}$ is an associated prime of $\mrm{H}^{\inf(M)}(M)_{\bar{\p}}$.
Moreover, according to \cite[Tag 0586]{SP},
we have that $\bar{\p} \in \opn{Supp}_{\mrm{H}^0(A)}(\mrm{H}^{\inf(M)}(M))$.
Since
\[
\mrm{H}^{\inf(M)}(M_{\bar{\p}}) \cong \mrm{H}^{\inf(M)}(M)_{\bar{\p}},
\]
we deduce that $\bar{\p} \in \opn{Supp}_A(M)$,
and that
\[
\inf(M_{\bar{\p}}) = \inf(M).
\]
It follows from Proposition \ref{prop:depth-lower-bound}
that $\bar{\p} \in \opn{Ass}_A(M)$,
which gives a contradiction.
Hence, $\bar{x}$ is $M$-regular.
\end{proof}

Given a commutative ring $A$,
a complex of $A$-modules $M$,
and $x \in A$,
we denote by $K_A(x;M)$ the Koszul complex of $M$ with respect to $x$.
Explicitly, we have that
\[
K_A(x;A) = 0 \to A \xrightarrow{\cdot x} A \to 0, \quad K_A(x;M) = K_A(x;A) \otimes_A M
\]
where the complex $K_A(x;A)$ is concentrated in degrees $-1,0$.

\begin{lem}\label{lem:koszul-red}
Let $A$ be a commutative DG-ring,
let $M \in \cat{D}(A)$,
and let $\bar{x} \in \mrm{H}^0(A)$.
Then there is an isomorphism
\[
\left(M//\bar{x}M\right) \otimes^{\mrm{L}}_A \mrm{H}^0(A) \cong K_{\bar{A}}\left(\bar{x};M\otimes^{\mrm{L}}_A \mrm{H}^0(A)\right)
\]
in $\cat{D}(\mrm{H}^0(A))$.
\end{lem}
\begin{proof}
To compute the left hand side,
we apply the functor 
\[
-\otimes^{\mrm{L}}_A M \otimes^{\mrm{L}}_A \mrm{H}^0(A)
\]
to the distinguished triangle
\[
A \xrightarrow{\cdot \bar{x}} A \to A//\bar{x} \to A[1]
\]
in $\cat{D}(A)$.
We obtain a distinguished triangle
\[
M \otimes^{\mrm{L}}_A \mrm{H}^0(A) \to M \otimes^{\mrm{L}}_A \mrm{H}^0(A) \to \left(M//\bar{x}M\right) \otimes^{\mrm{L}}_A \mrm{H}^0(A) \to \left(M \otimes^{\mrm{L}}_A \mrm{H}^0(A)\right)[1]
\]
in $\cat{D}(\mrm{H}^0(A))$.
Since the cone of the leftmost map in this triangle is exactly the Koszul complex
\[
K_{\bar{A}}\left(\bar{x};M\otimes^{\mrm{L}}_A \mrm{H}^0(A)\right)
\]
the result follows.
\end{proof}

\begin{lem}\label{lem:W0Aequal}
Let $(A,\bar{\m})$ be a noetherian local DG-ring,
and let $M \in \cat{D}^{-}(A)$. Then there is an equality of sets of prime ideals:
\[
W_0^A(M) = W_0^{\mrm{H}^0(A)}\left(M\otimes^{\mrm{L}}_A \mrm{H}^0(A)\right).
\]
\end{lem}
\begin{proof}
By Proposition \ref{prop:supp-red} there is an equality
\[
\opn{Supp}_A(M) = \opn{Supp}_{\mrm{H}^0(A)}\left(M\otimes^{\mrm{L}}_A \mrm{H}^0(A)\right).
\]
By Proposition \ref{prop:dim-reduction},
there is an equality
\[
\lcdim_A(M) = \lcdim_{\bar{A}}\left(M\otimes^{\mrm{L}}_A \mrm{H}^0(A)\right).
\]
By (\ref{eqn:red-ten}) and Lemma \ref{lem:local-ten-hz},
for any $\bar{\p} \in \opn{Spec}(\mrm{H}^0(A))$, there is an equality
\[
\sup(M_{\bar{\p}}) = \sup\left((M\otimes^{\mrm{L}}_A \mrm{H}^0(A))_{\bar{\p}} \right).
\]
These three equalities and the definition of $W_0$ imply the result.
\end{proof}

\begin{prop}\label{prop:out-of-w0-dim}
Let $(A,\bar{\m})$ be a noetherian local DG-ring,
let $M \in \cat{D}^{-}_{\mrm{f}}(A)$,
let $\bar{x} \in \bar{m} \subseteq \mrm{H}^0(A)$,
and assume that
\[
\bar{x} \notin \bigcup_{\bar{\p} \in W_0^A(M)} \bar{\p}.
\]
Then
\[
\lcdim(M//\bar{x}M) = \lcdim(M) - 1.
\]
\end{prop}
\begin{proof}
By Proposition \ref{prop:dim-reduction} we have that
\[
\lcdim_A(M//\bar{x}M) = \lcdim_{\bar{A}}\left(M//\bar{x}M \otimes^{\mrm{L}}_A \mrm{H}^0(A)\right).
\]
By Lemma \ref{lem:koszul-red} this is equal to
\[
\lcdim_{\bar{A}} \left(K_{\mrm{H}^0(A)}\left(\bar{x};M\otimes^{\mrm{L}}_A \mrm{H}^0(A)\right) \right).
\]
It follows from Lemma \ref{lem:W0Aequal} that
\[
\bar{x} \notin \bigcup_{\bar{\p} \in W_0^{\mrm{H}^0(A)}\left(M\otimes^{\mrm{L}}_A \mrm{H}^0(A)\right)} \bar{\p}.
\]
Hence, by \cite[Proposition 2.8]{ChII}, there is an equality
\[
\lcdim_{\bar{A}} \left(K_{\mrm{H}^0(A)}\left(\bar{x};M\otimes^{\mrm{L}}_A \mrm{H}^0(A)\right) \right)
= \lcdim_{\bar{A}} \left(M\otimes^{\mrm{L}}_A \mrm{H}^0(A)\right) -1
\]
and by Proposition \ref{prop:dim-reduction}, 
the latter is equal to $\lcdim_A(M) - 1$, as claimed.
\end{proof}

\begin{lem}\label{lem:dimanddepth}
Let $(A,\bar{\m})$ be a noetherian local DG-ring,
let $M \in \cat{D}^{\mrm{b}}_{\mrm{f}}(A)$,
and suppose that:
\begin{enumerate}
\item $\opn{seq.depth}_A(M) > 0$; that is $\depth_A(M) > \inf(M)$
\item $\opn{der.dim}_A(M) > 0$; that is $\lcdim_A(M) > \sup(M)$.
\end{enumerate}
Then there exists $\bar{x} \in \bar{\m}$ such that $\bar{x}$ is $M$-regular,
and moreover $\lcdim(M//\bar{x}M) = \lcdim(M)-1$.
\end{lem}
\begin{proof}
It follows from assumption (1) that $\bar{\m} \notin \opn{Ass}_A(M)$,
and from assumption (2) that $\bar{\m} \notin W_0^A(M)$.
By Proposition \ref{prop:assfinite} and Proposition \ref{prop:W0finite},
the two sets of prime ideals $\opn{Ass}_A(M), W_0^A(M)$ are finite sets,
and since they do not include $\bar{\m}$,
by the prime avoidance lemma,
there exists $\bar{x} \in \bar{\m}$ such that
\[
\bar{x} \notin \bigcup_{\bar{\p} \in \opn{Ass}_A(M)} \bar{\p},
\]
and moreover
\[
\bar{x} \notin \bigcup_{\bar{\p} \in W_0^A(M)} \bar{\p}.
\]
It follows from Proposition \ref{prop:out-of-ass-reg} that $\bar{x}$ is $M$-regular,
and from Proposition \ref{prop:out-of-w0-dim} that $\lcdim(M//\bar{x}M) = \lcdim(M)-1$,
as claimed.
\end{proof}

The following theorem is a DG version of \cite[Theorem 3.6]{ChII}.

\begin{thm}\label{thm:dg-sys-par}
Let $(A,\bar{\m})$ be a noetherian local DG-ring,
and let $M \in \cat{D}^{\mrm{b}}_{\mrm{f}}(A)$.
Assume that
\[
\amp\left(\mrm{R}\Gamma_{\bar{\m}}(M)\right) \ge \amp(M).
\]
Then there exists a maximal $M$-regular sequence $\bar{x}_1, \dots, \bar{x}_n \in \bar{\m}$
such that for each $1 \le i \le n$, 
there is an equality
\begin{equation}\label{eqn:in-thm-dim}
\lcdim\left(M//(\bar{x}_1,\dots,\bar{x}_i)M\right) = \lcdim(M) - i.
\end{equation}
\end{thm}
\begin{proof}
The proof is by induction on $\opn{seq.depth}_A(M)$.
If $\opn{seq.depth}_A(M) = 0$, 
there is nothing to prove, 
as in this case, a maximal $M$-regular sequence is empty.
Suppose $\opn{seq.depth}_A(M) > 0$;
that is, suppose that $\depth_A(M) > \inf(M)$.
By Theorem \ref{thm:non-vanish} and Proposition \ref{prop:depth},
we have that
\[
\amp\left(\mrm{R}\Gamma_{\bar{\m}}(M)\right) = \lcdim(M) - \depth(M).
\]
By assumption, this number is greater or equal to $\amp(M) = \sup(M) - \inf(M)$,
which implies that
\[
\opn{der.dim}(M) = \lcdim(M) - \sup(M) \ge \depth(M) - \inf(M) > 0.
\]
It follows that the conditions of Lemma \ref{lem:dimanddepth} are satisfied,
so there exists $\bar{x} \in \bar{\m}$ such that $\bar{x}$ is $M$-regular,
and moreover $\lcdim(M//\bar{x}M) = \lcdim(M)-1$.
Since $\bar{x}$ is $M$-regular,
we have that $\depth(M//\bar{x}M) = \depth(M)-1$. 
Also, by (\ref{eqn:reg-has-same-amp}) we have that $\amp(M//\bar{x}M) = \amp(M)$,
and $\inf(M//\bar{x}M) = \inf(M)$, 
so that 
\[
\opn{seq.depth}_A(M//\bar{x}M) = \opn{seq.depth}_A(M)-1.
\]
Applying Theorem \ref{thm:non-vanish} and Proposition \ref{prop:depth} again,
we obtain
\begin{align}
\amp\left(\mrm{R}\Gamma_{\bar{\m}}(M//\bar{x}M)\right) = \lcdim(M//\bar{x}M) - \depth(M//\bar{x}M) = \nonumber\\
= (\lcdim(M) - 1) - (\depth(M) - 1) = \nonumber\\
= \amp\left(\mrm{R}\Gamma_{\bar{\m}}(M)\right) \ge \amp(M) = \amp(M//\bar{x}M).\nonumber
\end{align}
By the induction hypothesis,
there is a maximal $M//\bar{x}M$-regular sequence $\bar{x}_2,\dots \bar{x}_n \in \bar{\m}$ such that 
\[
\lcdim\left( (M//\bar{x}M)//(\bar{x}_2,\dots \bar{x}_i) \right) =
\lcdim (M//\bar{x}M) -(i-1) = \lcdim(M) - i,
\]
which implies that $\bar{x}_1,\dots \bar{x}_n \in \bar{\m}$ is a maximal $M$-regular sequence which satisfies (\ref{eqn:in-thm-dim}).
\end{proof}

\begin{cor}\label{cor:DGhasSOP}
Let $(A,\bar{\m})$ be a noetherian local DG-ring with bounded cohomology.
Then there exists a maximal $A$-regular sequence $\bar{x}_1,\dots,\bar{x}_n \in \bar{\m}$
such that $\bar{x}_1,\dots,\bar{x}_n$ can be completed to system of parameters of $\mrm{H}^0(A)$.
\end{cor}
\begin{proof}
It follows from Theorem \ref{thm:main-amp}(1) that     
$A$ satisfies the assumption of Theorem \ref{thm:dg-sys-par},
so there exists a maximal $A$-regular sequence $\bar{x}_1,\dots,\bar{x}_n \in \bar{\m}$
which satisfies
\[
\lcdim\left(A//(\bar{x}_1,\dots,\bar{x}_i)\right) = \lcdim(A) - i.
\]
The result now follows from the fact that
\[
\lcdim\left(A//(\bar{x}_1,\dots,\bar{x}_i)\right) = \dim\left(\mrm{H}^0(A)/(\bar{x}_1,\dots,\bar{x}_i)\right),
\]
and that $\lcdim(A) = \dim\left(\mrm{H}^0(A)\right)$.
\end{proof}

\begin{cor}\label{cor:CMhasSP}
Let $(A,\bar{\m})$ be a noetherian local-Cohen-Macaulay DG-ring.
Then there exists a maximal $A$-regular sequence $\bar{x}_1,\dots,\bar{x}_n \in \bar{\m}$
which is a system of parameters of $\mrm{H}^0(A)$.
Moreover, for each $1\le i \le n$,
the DG-ring $A//(\bar{x}_1,\dots,\bar{x}_i)$ is local-Cohen-Macaulay.
\end{cor}
\begin{proof}
The first claim follows immediately from Corollary \ref{cor:DGhasSOP} 
since the fact that $A$ is local-Cohen-Macaulay implies that $\opn{seq.depth}(A) = \dim(\mrm{H}^0(A))$.
To prove the second claim, 
note that $\opn{seq.depth}(A//\bar{x}_1) = \opn{seq.depth}(A) - 1$,
and that $\dim(\mrm{H}^0(A//\bar{x}_1)) = \dim(\mrm{H}^0(A)) - 1$,
so by Corollary \ref{cor:chdepth}, 
the DG-ring $A//\bar{x}_1$ is local-Cohen-Macaulay,
and the general result follows by induction on $\opn{seq.depth}(A) = \dim(\mrm{H}^0(A))$.
\end{proof}

We now prove a DG-version of the Bass conjecture about local-Cohen-Macaulay rings.

\begin{thm}\label{thm:bassConj}
Let $(A,\bar{\m})$ be a noetherian local DG-ring with bounded cohomology.
\begin{enumerate}[wide, labelwidth=!, labelindent=0pt]
\item If $A$ is local-Cohen-Macaulay, there exists $0 \ncong M \in \cat{D}^{\mrm{b}}_{\mrm{f}}(A)$
such that $\injdim_A(M) < \infty$, and such that $\amp(M) = \amp(A)$.
\item Assume further that $A$ has a noetherian model.
For any $0 \ncong M \in \cat{D}^{\mrm{b}}_{\mrm{f}}(A)$
such that $\injdim_A(M) < \infty$,
we have that $\amp(M) \ge \amp(A)$.
If there exists such $M$ with $\amp(M) = \amp(A)$,
then $A$ is local-Cohen-Macaulay.
\end{enumerate}
\end{thm}
\begin{proof}
\begin{enumerate}[wide, labelwidth=!, labelindent=0pt]
\item Let $d = \dim(\mrm{H}^0(A))$.
By Corollary \ref{cor:CMhasSP},
there exists a system of parameters $\bar{x}_1,\dots,\bar{x}_d \in \bar{\m}$ of $\mrm{H}^0(A)$
which is a maximal $A$-regular sequence.
Consider the DG-module $N = A//(\bar{x}_1,\dots,\bar{x}_d) \in \cat{D}^{\mrm{b}}_{\mrm{f}}(A)$.
It follows from (\ref{eqn:reg-has-same-amp}) that $\amp(N) = \amp(A)$.
Moreover, note that by our construction of $N$,
we have that $\flatdim_A(N) < \infty$.
Let $E = E(A,\bar{\m}) \in \cat{D}(A)$ be the injective DG-module corresponding to the injective hull $\bar{E}$ of the residue field $\mrm{H}^0(A)/\bar{\m}$.
Finally, let $M = \mrm{R}\opn{Hom}_A(N,E)$.
Note that by \cite[Theorem 4.10]{Sh2},
we have for all $n\in \mathbb{Z}$ an isomorphism
\begin{equation}\label{eqn:mat-dual-is-hm}
\mrm{H}^n(M) \cong \opn{Hom}_{\mrm{H}^0(A)}(\mrm{H}^{-n}(N),\bar{E}),
\end{equation}
so since $\bar{E}$ is a cogenerator of $\opn{Mod}(\mrm{H}^0(A))$,
we have that $\mrm{H}^n(M) = 0$ if and only if $\mrm{H}^{-n}(N) = 0$.
We deduce that $\amp(M) = \amp(N)$.
Since $\flatdim_A(N) < \infty$ and $\injdim_A(E) < \infty$, 
it follows from the adjunction isomorphism
\[
\mrm{R}\opn{Hom}_A(-,M) = \mrm{R}\opn{Hom}_A(-,\mrm{R}\opn{Hom}_A(N,E)) \cong 
\mrm{R}\opn{Hom}_A(-\otimes^{\mrm{L}}_A N,E)
\]
that $\injdim_A(M) < \infty$.
It remains to show that the cohomologies of $M$ are finitely generated over $\mrm{H}^0(A)$.
To see this, note first that the fact that $\bar{x}_1,\dots,\bar{x}_d$ is a system of parameters of $\mrm{H}^0(A)$
implies that
\[
\mrm{H}^0(N) = \mrm{H}^0\left(A//(\bar{x}_1,\dots,\bar{x}_d)\right) = 
\mrm{H}^0(A)/(\bar{x}_1,\dots,\bar{x}_d)
\]
is a zero-dimensional local ring.
For any $n \in \mathbb{Z}$, we have that $\mrm{H}^{-n}(N)$ is a finitely generated $\mrm{H}^0(N)$-module,
and since the $\mrm{H}^0(A)$-action on $\mrm{H}^{-n}(N)$ factors through the ring $\mrm{H}^0(N)$,
we deduce that the $\mrm{H}^0(A)$-module $\mrm{H}^{-n}(N)$ is artinian,
and hence of finite length.
It follows from Matlis duality that its Matlis dual,
which by (\ref{eqn:mat-dual-is-hm}) is $\mrm{H}^{n}(M)$ is finitely generated, 
proving the claim.
\item Since $A$ has a noetherian model, $M \in \cat{D}^{\mrm{b}}_{\mrm{f}}(A)$ and $\injdim_A(M) < \infty$,
it follows from \cite[Theorem B]{Jo2} that 
\[
\amp(M) \ge \amp\left(\mrm{R}\Gamma_{\bar{\m}}(A)\right).
\]
By Theorem \ref{thm:main-amp}(1), 
we have that 
\[
\amp\left(\mrm{R}\Gamma_{\bar{\m}}(A)\right) \ge \amp(A),
\]
which implies that $\amp(M) \ge \amp(A)$.
These inequalities also imply that if $\amp(M) = \amp(A)$,
then we must have 
\[
\amp\left(\mrm{R}\Gamma_{\bar{\m}}(A)\right) = \amp(A),
\]
so that $A$ is local-Cohen-Macaulay.
\end{enumerate}
\end{proof}

\begin{rem}
Item (2) above solves a recent conjecture of Minamoto under the mild noetherian model assumption.
See \cite[Conjecture 3.36]{Mi2}.
\end{rem}

\begin{rem}\label{rem:noet-model}
Unlike the rest of this paper, in item (2) above we had to impose the noetherian model assumption,
in addition to our standing assumption that DG-rings are noetherian.
The reason for this is our use of the results of \cite{Jo1,Jo2} which made this assumption.
We conjecture that this assumption is redundant, 
both here and in general, in the main theorems of \cite{Jo1,Jo2}.
\end{rem}

\section{Local-Cohen-Macaulay DG-modules}\label{sec:MCM}

In this section we will define and study local-Cohen-Macaulay DG-modules
and maximal local-Cohen-Macaulay DG-modules over local-Cohen-Macaulay DG-rings.
Over a noetherian local ring, one can define both Cohen-Macaulay modules and,
more generally, Cohen-Macaulay complexes. 
Our notion of a local-Cohen-Macaulay DG-module generalizes Cohen-Macaulay modules and not Cohen-Macaulay complexes.
As Yekutieli and Zhang observed in \cite[Theorem 6.2]{YZ}, 
the category of Cohen-Macaulay complexes over a ring is an abelian category,
so this makes finding a definition which generalizes Cohen-Macaulay complexes particularly appealing, 
but it is currently not clear to us how to do that.

For Cohen-Macaulay modules, however, we do have the following generalization:

\begin{dfn}\label{def:CMM}
Let $(A,\bar{\m})$ be a noetherian local DG-ring with $\amp(A) < \infty$.
We say that $M \in \cat{D}^{\mrm{b}}_{\mrm{f}}(A)$ is a local-Cohen-Macaulay DG-module if there are equalities
\[
\amp(M) = \amp(A) = \amp\left(\mrm{R}\Gamma_{\bar{\m}}(M)\right).
\]
We denote by $\opn{CM}(A)$ the full subcategory of $\cat{D}^{\mrm{b}}_{\mrm{f}}(A)$ consisting of local-Cohen-Macaulay DG-module.
\end{dfn}

It is clear from this definition and Definition \ref{def:CM} 
that $A$ is local-Cohen-Macaulay as a DG-module if and only if it is local-Cohen-Macaulay as a DG-ring.

\begin{prop}\label{CMbyDual}
Let $(A,\bar{\m})$ be a noetherian local-Cohen-Macaulay DG-ring,
and let $R$ be a dualizing DG-module over $A$.
Let $M \in \cat{D}^{\mrm{b}}_{\mrm{f}}(A)$ be such that $\amp(M) = \amp(A)$.
Then $M \in \opn{CM}(A)$ if and only if 
\[
\amp\left(\mrm{R}\opn{Hom}_A(M,R)\right) = \amp(A).
\]
In this case we also have that 
\[
\mrm{R}\opn{Hom}_A(M,R) \in \opn{CM}(A).
\]
\end{prop}
\begin{proof}
Let $\bar{E}$ be the $\mrm{H}^0(A)$-module which is the injective hull of $\mrm{H}^0(A)/\bar{\m}$.
By the DG local duality theorem (\cite[Theorem 7.26]{Sh2}),
we have that
\[
\mrm{H}^n\left( \mrm{R}\Gamma_{\bar{\m}}(M) \right) = \opn{Hom}_{\mrm{H}^0(A)}\left(\opn{Ext}^{-n}_A(M,R),\bar{E}\right).
\]
Since $\bar{E}$ is an injective cogenerator, 
we deduce that
\[
\amp\left( \mrm{R}\Gamma_{\bar{\m}}(M) \right) = \amp\left(\mrm{R}\opn{Hom}_A(M,R)\right),
\]
proving the first claim.
The second claim follows from the duality isomorphism
\[
M \cong \mrm{R}\opn{Hom}_A(\mrm{R}\opn{Hom}_A(M,R),R)
\]
that holds because $R$ is a dualizing DG-module.
\end{proof}

\begin{prop}
Let $(A,\bar{\m})$ be a noetherian local-Cohen-Macaulay DG-ring,
and let $R$ be a dualizing DG-module over $A$.
Then $R \in \opn{CM}(A)$.
\end{prop}
\begin{proof}
By Proposition \ref{prop:CM-by-DC} we have that $\amp(R) = \amp(A)$,
and since $R$ is dualizing, 
\[
\mrm{R}\opn{Hom}_A(R,R) = A,
\]
so the result follows from Proposition \ref{CMbyDual}.
\end{proof}

In view of the local duality theorem, the above duality of local-Cohen-Macaulay DG-modules might seem tautological.
More interesting is its restriction to maximal local-Cohen-Macaulay DG-modules which we now discuss.

Recall that over a noetherian local ring $A$,
a Cohen-Macaulay $A$-module $M$ is called maximal Cohen-Macaulay if $\dim(M) = \dim(A)$.
To generalize this to local DG-rings, we recall from (\ref{eqn:dim})
that if $A$ is a $(A,\bar{\m})$ be a noetherian local DG-ring,
and $M \in \cat{D}^{\mrm{b}}_{\mrm{f}}(A)$
that $\lcdim(M) \le \sup(M) + \dim(\mrm{H}^0(A))$.
It is thus make sense to define:

\begin{dfn}\label{def:MCMM}
Let $(A,\bar{\m})$ be a noetherian local DG-ring with $\amp(A) < \infty$,
and let $M \in \opn{CM}(A)$.
We say that $M$ is a maximal local-Cohen-Macaulay DG-module if $\lcdim(M) = \sup(M) + \dim(\mrm{H}^0(A))$.
We denote by $\opn{MCM}(A)$ the full subcategory of $\opn{CM}(A)$ consisting of maximal local-Cohen-Macaulay DG-modules over $A$.
\end{dfn}

It follows from this definition and from (\ref{eqn:dimofA}) that if $A$ is a local-Cohen-Macaulay DG-ring,
then $A \in \opn{MCM}(A)$. 

We will soon prove that dualizing DG-modules over a local-Cohen-Macaulay DG-ring are also maximal Cohen-Macaulay DG-modules.
Before proving this, let us discuss the situation over rings.
Let $A$ be a local ring, 
and let a $R$ be a dualizing complex over $A$,
normalized so that $\inf(R) = -\dim(A)$.
According to \cite[tag 0A7U]{SP}, for any $0 \le i \le \dim(A)$ we have that
\[
\dim\left( \mrm{H}^{-i}(R) \right) \le i
\]
and $\mrm{H}^i(R) = 0$ for $i \notin [-d,0]$.
Moreover, by \cite[tag 0AWE]{SP}, 
there is an equality 
\[
\dim\left( \mrm{H}^{-d}(R) \right) = d.
\]
It follows from these facts and the definition of local cohomology Krull dimension that 
\begin{equation}\label{eqn:dim-of-dc}
\lcdim(R) = 0 = \inf(R) + \dim(A).
\end{equation}
Hence, it happens that $\lcdim(R) = \sup(R) + \dim(A)$ if and only if $\amp(R) = 0$,
if and only if $A$ is a Cohen-Macaulay ring. 
Let us record this fact, as we are not aware of any literature containing it.
\begin{prop}
Let $(A,\m)$ be a noetherian local ring of dimension $d$,
and let $R$ be a dualizing complex over $A$.
Then $A$ is Cohen-Macaulay if and only if 
\[
\lcdim(R) = \sup(R) + d.
\]
\end{prop}
\qed

The next result generalizes this discussion to DG-rings.
In the DG-setting, unlike the case of rings,
there can be more than one cohomology module of $R$ of dimension $d$
(and there are at least two such cohomologies if $A$ is local-Cohen-Macaulay and $\amp(A) > 0$).

\begin{thm}\label{thm:str-of-dc}
Let $(A,\bar{\m})$ be a noetherian local DG-ring,
let $d = \dim(\mrm{H}^0(A))$,
and let $R$ be a dualizing DG-module over $A$.
Assume that $R$ is normalized so that $\inf(R) = -d$.
Then the following holds:
\begin{enumerate}
\item
There is an equality
\[
\dim_{\bar{A}}\left(\mrm{H}^{\inf(R)}(R)\right) = \dim_{\bar{A}}\left(\mrm{H}^0(A)\right).
\]  
\item If moreover $\amp(A) = n < \infty$,
then for every $0 \le i \le \dim_{\bar{A}}(\mrm{H}^0(A))$ there is an inequality
\[
\dim_{\bar{A}}\left(\mrm{H}^{-i+n}(R)\right) \le i
\]
\item Assuming $\amp(A) = n < \infty$,
we have that $A$ is local-Cohen-Macaulay if and only if there is an equality
\[
\dim_{\bar{A}}\left(\mrm{H}^{\sup(R)}(R)\right) = \dim_{\bar{A}}\left(\mrm{H}^0(A)\right),
\]
if and only if
\[
\lcdim_A(R) = \sup(R) + \dim(\mrm{H}^0(A)),
\]
if and only if $R \in \opn{MCM}(A)$.
\end{enumerate}
\end{thm}
\begin{proof}
\begin{enumerate}[wide, labelwidth=!, labelindent=0pt]
\item
By assumption $\inf(R) = -d$.
According to (\ref{eqn:red-rhom}), we have that
\[
\inf\left(\mrm{R}\opn{Hom}_A(\mrm{H}^0(A),R)\right) = -d, \quad
\mrm{H}^{-d}(R) = \mrm{H}^{-d}\left(\mrm{R}\opn{Hom}_A(\mrm{H}^0(A),R)\right).
\]
Hence, it is enough to show that
\[
\dim_{\bar{A}}\left(\mrm{H}^{-d}\left(\mrm{R}\opn{Hom}_A(\mrm{H}^0(A),R)\right)\right) = d,
\]
and this follows from \cite[tag 0AWE]{SP},
because by \cite[Proposition 7.5]{Ye1}, 
the complex $\mrm{R}\opn{Hom}_A(\mrm{H}^0(A),R)$ is a dualizing complex over $\mrm{H}^0(A)$.
\item According to \cite[Proposition 7.25]{Sh2},
we have that 
\[
\mrm{R}\Gamma_{\bar{\m}}(R) = E(A,\bar{\m}),
\]
so in particular 
\[
\sup\left(\mrm{R}\Gamma_{\bar{\m}}(R)\right) = n.
\]
Hence, by Theorem \ref{thm:non-vanish}, 
we deduce that $\lcdim(R) = n$.
Given $0 \le i \le \dim_{\bar{A}}(\mrm{H}^0(A))$,
by the definition of Krull dimension of DG-modules,
there are inequalities
\[
n = \lcdim(R) = \sup_{\ell \in \mathbb{Z}} \left\{\dim(\mrm{H}^{\ell}(R)) + \ell\right\}  \ge \dim_{\bar{A}}\left(\mrm{H}^{-i+n}(R)\right) + (-i+n)
\]
which implies that
\[
\dim_{\bar{A}}\left(\mrm{H}^{-i+n}(R)\right) \le i.
\]
\item The fact that 
\[
\dim_{\bar{A}}\left(\mrm{H}^{\sup(R)}(R)\right) = \dim_{\bar{A}}\left(\mrm{H}^0(A)\right),
\]
if and only if
\[
\lcdim_A(R) = \sup(R) + \dim(\mrm{H}^0(A)),
\]
follows from the definition of dimension and from (\ref{eqn:max-dim}).
To see that this is equivalent to $A$ being local-Cohen-Macaulay,
note that, as we have seen in the proof of (2), $\lcdim_A(R) = n$.
It follows that 
\[
\lcdim_A(R) = \sup(R) + \dim(\mrm{H}^0(A))
\]
if and only if $\sup(R) = n-d$,
if and only if 
\[
\amp(R) = \sup(R)-\inf(R) = (n-d)-(-d) = n,
\]
and by Proposition \ref{prop:CM-by-DC},
this is equivalent to $A$ being local-Cohen-Macaulay.
Finally, if $A$ is local-Cohen-Macaulay, 
we have seen that $R \in \opn{CM}(A)$,
so the above implies that $R \in \opn{MCM}(A)$,
while if $R \in \opn{MCM}(A)$,
in particular $R \in \opn{CM}(A)$,
so that $A$ is local-Cohen-Macaulay.
\end{enumerate}
\end{proof}

\begin{rem}\label{rem:DCvanCoh}
In view of the above result, 
it is natural to wonder if for a dualizing DG-module $R$ over a local-Cohen-Macaulay DG-ring $A$,
one has that $\dim(\mrm{H}^i(R)) = d$ for all $-d\le i\le n-d$.
This is not the case. In fact, it could happen that $\mrm{H}^i(R) = 0$ for all $-d<i<n-d$.
Such an example will be given in Example \ref{exa:DCvanCoh} below.
\end{rem}

As we have seen above, 
the topmost and bottommost cohomologies of a dualizing DG-module over a noetherian local-Cohen-Macaulay DG-ring
have maximal dimension. 
The next result shows that this is the case for every maximal local-Cohen-Macaulay DG-module.

\begin{prop}
Let $(A,\bar{\m})$ be a noetherian local DG-ring with $\amp(A) < \infty$,
and let $M \in \opn{MCM}(A)$.
Then
\[
\dim_{\bar{A}}\left(\mrm{H}^{\inf(M)}(M)\right) = \dim_{\bar{A}}\left(\mrm{H}^{\sup(M)}(M)\right) = \dim(\mrm{H}^0(A)).
\]
\end{prop}
\begin{proof}
The second equality follows from the definition and from (\ref{eqn:max-dim}).
Since $M$ is maximal local-Cohen-Macaulay, 
we have that 
\[
\depth_A(M) = \lcdim(M) - \sup(M) + \inf(M) = \dim\left(\mrm{H}^0(A)\right) + \inf(M),
\]
so by Proposition \ref{prop:depth-dg-bound} we deduce that
\[
\dim\left(\mrm{H}^0(A)\right) + \inf(M) \le \dim_{\bar{A}}\left(\mrm{H}^{\inf(M)}(M)\right) + \inf(M),
\]
which implies that 
\[
\dim\left(\mrm{H}^0(A)\right) \le \dim_{\bar{A}}\left(\mrm{H}^{\inf(M)}(M)\right),
\]
so these numbers must be equal.
\end{proof}

We now discuss characterizations of local-Cohen-Macaulay and maximal local-Cohen-Macaulay DG-modules using the notion of a regular sequence from the previous section.

\begin{prop}
Let $(A,\bar{\m})$ be a noetherian local DG-ring,
let $M \in \cat{D}^{\mrm{b}}_{\mrm{f}}(A)$,
and suppose that $\amp(M) = \amp(A)$.
Then $M \in \opn{CM}(A)$ if and only if 
the length of a maximal $M$-regular sequence contained in $\bar{\m}$ is equal to $\lcdim(M) - \sup(M)$;
that is if and only if $\opn{seq.depth}_A(M) = \opn{der.dim}_A(M)$.
If $M \in \opn{CM}(A)$,
then $M \in \opn{MCM}(A)$ if and only if $\opn{seq.depth}_A(M) = \opn{der.dim}(A)$.
\end{prop}
\begin{proof}
It follows from the definitions and from from Theorem \ref{thm:non-vanish} and Proposition \ref{prop:depth} that
$\opn{seq.depth}_A(M) = \opn{der.dim}_A(M)$ if and only if $\amp(\mrm{R}\Gamma_{\bar{\m}}(M)) = \amp(M)$ if and only if $M \in \opn{CM}(A)$.
In that case, it also follows from the definitions that $M \in \opn{MCM}(A)$ if and only if $\opn{seq.depth}_A(M) = \opn{der.dim}(A)$.
\end{proof}

Next we show that like Cohen-Macaulay DG-modules, 
over local-Cohen-Macaulay DG-rings maximal local-Cohen-Macaulay DG-modules are also self-dual:

\begin{prop}
Let $(A,\bar{\m})$ be a noetherian local-Cohen-Macaulay DG-ring,
let $R$ be a dualizing DG-module over $A$, 
and let $M \in \opn{MCM}(A)$.
Then 
\[
\mrm{R}\opn{Hom}_A(M,R) \in \opn{MCM}(A).
\]
\end{prop}
\begin{proof}
It is clear that $\mrm{R}\opn{Hom}_A(M,R) \in \opn{CM}(A)$,
so we only need to show that
\[
\lcdim\left(\mrm{R}\opn{Hom}_A(M,R)\right) - \sup\left(\mrm{R}\opn{Hom}_A(M,R)\right) = \dim(\mrm{H}^0(A)).
\]
According to Theorem \ref{thm:non-vanish} we have that 
\[
\lcdim\left(\mrm{R}\opn{Hom}_A(M,R)\right) = \sup\left(\mrm{R}\Gamma_{\bar{\m}}(\mrm{R}\opn{Hom}_A(M,R))\right),
\]
and by the DG local duality theorem (\cite[Theorem 7.26]{Sh2}),
the latter is equal to $-\inf(M)$.
The local duality theorem also implies that
\[
\sup\left(\mrm{R}\opn{Hom}_A(M,R)\right) = -\inf\left(\mrm{R}\Gamma_{\bar{\m}}(M)\right),
\]
and by Proposition \ref{prop:depth} the latter is equal to $-\depth(M)$.
Hence,
\begin{align}
\lcdim\left(\mrm{R}\opn{Hom}_A(M,R)\right) - \sup\left(\mrm{R}\opn{Hom}_A(M,R)\right) =\nonumber\\
-\inf(M)+\depth(M) = \opn{seq.depth}(M)\nonumber
\end{align}
and since $M$ is maximal local-Cohen-Macaulay, 
the latter is equal to $\dim(\mrm{H}^0(A))$, 
as claimed.
\end{proof}

\section{Trivial extension DG-rings and the local-Cohen-Macaulay property}\label{sec:text}

In this section we study trivial extension DG-rings.
This will allow us to construct many examples of local-Cohen-Macaulay DG-rings.
We first recall the construction, 
following \cite[Section 1]{Jo3} (but note that we use cohomological notation, unlike the homological notation used there).

Let $A$ be a commutative DG-ring,
and let $M \in \cat{D}(A)$. Suppose that $\sup(M) < 0$.
Donating by $d_A$ the differential of $A$ and by $d_M$ the differential of $M$,
we give the graded abelian group $A\oplus M$ the structure of a commutative DG-ring by letting the differential be
\[
d_{A \oplus M}\left( \begin{matrix} a \\ m \end{matrix} \right) = \left(\begin{matrix} d_A(a) \\ d_M(m) \end{matrix} \right),
\]
and defining multiplication by the rule
\[
\left(\begin{matrix} a_1 \\ m_1 \end{matrix} \right) \cdot \left(\begin{matrix} a_2 \\ m_2 \end{matrix} \right) = 
\left(\begin{matrix} a_1\cdot a_2 \\ a_1 \cdot m_2 + m_1 \cdot a_2 \end{matrix} \right).
\]
This gives $A\oplus M$ the structure of a commutative (non-positive) DG-ring,
which we will denote by $A \skewtimes M$.
We note that there are natural maps of DG-rings $A \to A \skewtimes M \to A$,
such that their composition is equal to $1_A$.

The next result follows from the definition:
\begin{prop}\label{prop:triv-ext}
Let $A$ be a commutative noetherian DG-ring,
let $M \in \cat{D}^{-}_{\mrm{f}}(A)$,
and assume that $\sup(M) < 0$.
Then $A \skewtimes M$ is a commutative noetherian DG-ring with 
\[
\mrm{H}^0(A \skewtimes M) = \mrm{H}^0(A).
\]
In particular, if $(A,\bar{\m})$ is a local DG-ring,
then $(A \skewtimes M,\bar{\m})$ is a local DG-ring.
\end{prop}
\qed

This construction will allow us to construct many interesting examples and counterexamples.
For instance, let us show that the notion of a regular sequence behaves different in the DG-setting:

\begin{exa}\label{exa:reg-not-par}
Let $\k$ be a field,
let $R = \k[[x,y]]/(x \cdot y)$,
and let $M$ be the $R$-module $R/(x) \cong \k[[y]]$.
Consider the DG-ring $A = R \skewtimes M[1]$,
and consider $y \in \mrm{H}^0(A) = R$.
Since $y$ is $M$-regular,
and $M = \mrm{H}^{\inf(A)}(A)$,
we see that $y$ is $A$-regular.
However, since $\dim(R) = \dim(R/y)$,
we have that $\dim(\mrm{H}^0(A)) = \dim(\mrm{H}^0(A//y))$.
\end{exa}

Next, we provide a sufficient condition for the trivial extension to be local-Cohen-Macaulay:

\begin{thm}\label{thm:triv-ext}
Let $(A,\m)$ be a noetherian local ring, and $M \in \cat{D}^{\mrm{b}}_{\mrm{f}}(A)$.
Assume that:
\begin{enumerate}
\item $\sup(M) < 0$.
\item $\lcdim(M) = \inf(M) + \dim(A)$.
\item $\lcdim(M) \le \opn{depth}(A)$.
\end{enumerate}
Then the trivial extension $A \skewtimes M$ is a noetherian local-Cohen-Macaulay DG-ring if and only if $M$ is a Cohen-Macaulay complex over $A$. 
\end{thm}
\begin{proof}
By Proposition \ref{prop:triv-ext} and the definition,
 the DG-ring $A \skewtimes M$ is a noetherian local DG-ring with bounded cohomology,
and it holds that $\amp(A \skewtimes M) = -\inf(M)$.
To compute 
\[
\amp\left(\mrm{R}\Gamma_{\m}^{A \skewtimes M}(A \skewtimes M)\right),
\]
we may apply the forgetful functor $\cat{D}(A \skewtimes M) \to \cat{D}(A)$,
and since local cohomology commutes with it, 
and the ideals of definition of $A$ and $A \skewtimes M$ coincide,
we may compute instead the number
\[
\amp\left(\mrm{R}\Gamma_{\m}^{A}(A \skewtimes M)\right),
\]
where we now compute local cohomology over the local ring $A$.
Since in the category $\cat{D}(A)$ we have that $A \skewtimes M = A \oplus M$,
so that
\[
\mrm{R}\Gamma^A_{\m}(A \skewtimes M) \cong \mrm{R}\Gamma^A_{\m}(A) \oplus \mrm{R}\Gamma^A_{\m}(M).
\]
Since $\lcdim(M) \le \depth(A)$,
we must have $\depth(M) \le \depth(A)$.
Hence,
\[
\inf \left(\mrm{R}\Gamma^A_{\m}(A \skewtimes M) \right) = \inf \left(\mrm{R}\Gamma^A_{\m}(M)\right) = \depth(M).
\]
Similarly, since $\lcdim(M) = \inf(M) + \dim(A) \le \sup(M) + \dim(A) < \dim(A)$,
we have that
\[
\sup \left(\mrm{R}\Gamma^A_{\m}(A \skewtimes M) \right) = \sup \left(\mrm{R}\Gamma^A_{\m}(A)\right) = \dim(A) = \lcdim(M) - \inf(M).
\]
Hence,
\begin{align}
\amp\left(\mrm{R}\Gamma_{\m}^{A}(A \skewtimes M)\right) = \lcdim(M) - \inf(M) - \depth(M) =\nonumber\\
\left(\lcdim(M) - \depth(M)\right) + (-\inf(M)) =\nonumber\\
\left(\lcdim(M) - \depth(M)\right) + \amp(A \skewtimes M)\nonumber
\end{align}
We see that $A \skewtimes M$ is local-Cohen-Macaulay if and only if $\lcdim(M) - \depth(M) = 0$
if and only if $M$ is a Cohen-Macaulay complex over $A$.
\end{proof}

\begin{exa}
Let $(A,\m)$ be a noetherian local ring,
let $M$ be a finitely generated $A$-module,
and suppose that $\dim(M) = \dim(A)$.
Then $M[\dim(A)]$ satisfies the conditions of Theorem \ref{thm:triv-ext},
so $A\skewtimes M[\dim(A)]$ is a local-Cohen-Macaulay DG-ring if and only if $M$ is a Cohen-Macaulay module 
if and only if $M$ is a maximal Cohen-Macaulay module.
\end{exa}

\begin{exa}
A bit more generally,
if $(A,\m)$ is a noetherian local ring,
and $M$ is an arbitrary finitely generated $A$-module,
we can replace $A$ with $A/\opn{ann}(M)$,
and consider the DG-ring
\[
B = A/\opn{ann}(M)\skewtimes M[\dim(M)].
\]
Then the conditions of Theorem \ref{thm:triv-ext} hold,
and we see that $B$ is a local-Cohen-Macaulay DG-ring if and only if $M$ is a Cohen-Macaulay module over $A$.
\end{exa}

\begin{exa}
Let $(A,\m)$ be a noetherian local ring,
let $R$ be a dualizing complex over $A$,
and suppose that $\sup(R) < 0$.
J{\o}rgensen proved in \cite{Jo3} that in this case $A\skewtimes R$ is a Gorenstein DG-ring.
In particular, by Proposition \ref{prop:GorisCM}, 
$A\skewtimes R$ is a local-Cohen-Macaulay DG-ring.
Here is an alternative proof of the latter fact:
By (\ref{eqn:dim-of-dc}), condition (2) of Theorem \ref{thm:triv-ext} is satisfied.
Moreover, Grothendieck's local duality implies that:
\[
0 > \sup(R) = \inf(R) + \amp(R) = \inf(R) + \dim(A) - \depth(A) = \lcdim(R) - \depth(A),
\]
so that $\lcdim(R) < \depth(A)$.
Hence, all the conditions of Theorem \ref{thm:triv-ext} are satisfied,
and we deduce that $A\skewtimes R$ is a local-Cohen-Macaulay DG-ring.
\end{exa}

\begin{exa}\label{exa:DCvanCoh}
Following Remark \ref{rem:DCvanCoh},
let $(A,\m)$ be a Cohen-Macaulay local ring with a dualizing module $M$.
Then by the above mentioned Theorem of J{\o}rgensen from \cite{Jo3},
we have that $B = A \skewtimes M[\dim(A)]$ is a Gorenstein DG-ring.
In particular, $B$ is a local-Cohen-Macaulay DG-ring,
and a dualizing DG-module over itself. 
Moreover, for all $\inf(B) < i < \sup(B)$, we have that $\mrm{H}^i(B) = 0$.
\end{exa}

\begin{exa}\label{exa:no-dualizing}
Let $(A,\m)$ be a noetherian local ring of Krull dimension $d>0$ which does not have a dualizing complex,
but does have a maximal Cohen-Macaulay module $M$.
To give a concrete example of such $A,M$,
in \cite[Section 6, Example 1]{Og} there is an example of Cohen-Macaulay ring which does not have a canonical module.
Then one can take this ring to be $A$, and $M = A$.
By Theorem \ref{thm:triv-ext}, 
the DG-ring $B = A\skewtimes M[d]$ is a local-Cohen-Macaulay DG-ring.
Since $\mrm{H}^0(B) = A$ does not have a dualizing complex,
it follows from \cite[Proposition 7.5]{Ye1} that $B$ does not have a dualizing DG-module.
\end{exa}

\section{Finite maps, localizations and global Cohen-Macaulay DG-rings}\label{sec:localiz}

In this section we discuss analogues of the following two basic properties of Cohen-Macaulay rings and modules:

\begin{enumerate}[wide, labelwidth=!, labelindent=0pt]
\item (\textbf{Independence of base ring}):
Assume $f:(A,\m) \to (B,\n)$ is a local map between two noetherian local rings which is a finite ring map,
and that $M$ is a finitely generated $B$-module.
Then $M$ is Cohen-Macaulay over $B$ if and only if $M$ is Cohen-Macaulay over $A$.
\item (\textbf{Localization}): Let $(A,\m)$ be a noetherian local Cohen-Macaulay ring,
and let $\p \in \opn{Spec}(A)$.
Then $(A_{\p},\p\cdot A_{\p})$ is also Cohen-Macaulay.
\end{enumerate}

Unfortunately, both of these properties fail in general in the DG-case,
and for the same reason: change of amplitude.
Any commutative ring $A$ has $\amp(A) = 0$,
but different DG-rings have different amplitude.
Regarding item (1) above, if one assumes that $\amp(A) = \amp(B)$,
then its DG-generalization is true.
Regarding item (2), it could happen that for a DG-ring $A$,
and $\bar{\p} \in \opn{Spec}(\mrm{H}^0(A))$ there is a strict inequality $\amp(A_{\bar{\p}}) < \amp(A)$,
and when that happens, it could happen that $A$ is local-Cohen-Macaulay but $A_{\bar{\p}}$ is not.
Again, when $\amp(A_{\bar{\p}}) = \amp(A)$,
it holds that if $A$ is local-Cohen-Macaulay then $A_{\bar{\p}}$ is.
On the positive side, 
we will see below that often for local-Cohen-Macaulay DG-rings, 
there is always an equality $\amp(A_{\bar{\p}}) = \amp(A)$,
so that under the assumption that the topological space $\opn{Spec}(\mrm{H}^0(A))$ is irreducible,
the local-Cohen-Macaulay property is stable under localization.

\begin{prop}\label{prop:ind}
(Independence of base DG-ring):
Let $(A,\bar{\m})$ and $(B,\bar{\n})$ be two noetherian local DG-rings,
such that $\amp(A) = \amp(B) < \infty$.
Let $f:A \to B$ be a map of DG-rings such that the induced map
$\mrm{H}^0(f):\mrm{H}^0(A) \to \mrm{H}^0(B)$ is a finite ring map which is local.
Given $M \in \cat{D}^{\mrm{b}}_{\mrm{f}}(B)$,
we have that $M$ is a local-Cohen-Macaulay DG-module over $B$
if and only if $M$ is a local-Cohen-Macaulay DG-module over $A$.
\end{prop}

Before proving this result, 
we shall need the following lemma.

\begin{lem}\label{lem:RGofForg}
Let $(A,\bar{\m})$ and $(B,\bar{\n})$ be two noetherian local DG-rings,
let $f:A \to B$ be a map of DG-rings,
and suppose that $\mrm{H}^0(f)(\bar{\m})\cdot \mrm{H}^0(B)$ contains some power of $\bar{\n}$.
Letting $Q:\cat{D}(B) \to \cat{D}(A)$ be the forgetful functor,
there is an isomorphism
\[
\mrm{R}\Gamma_{\bar{\m}}\circ Q(-) \cong Q \circ \mrm{R}\Gamma_{\bar{\n}}(-)
\]
of functors $\cat{D}(B) \to \cat{D}(A)$.
\end{lem}
\begin{proof}
Let $\mathbf{a}$ be a finite sequence of elements of $A^0$ whose image in $\mrm{H}^0(A)$ generates $\bar{\m}$,
let $\mathbf{b}$ be a finite sequence of elements of $B^0$ whose image in $\mrm{H}^0(B)$ generates $\bar{\n}$,
and let $N \in \cat{D}(N)$.
By (\ref{eqn:RGamma}) there are natural isomorphisms:
\[
\mrm{R}\Gamma_{\bar{\m}}\left(Q(N)\right) \cong
\mrm{R}\Gamma_{\bar{\m}}(B\otimes^{\mrm{L}}_B N) \cong
\opn{Tel}(A^0;\mathbf{a}) \otimes_{A^0} A \otimes^{\mrm{L}}_A B\otimes^{\mrm{L}}_B N
\]
By associativity of the derived tensor product, 
this is naturally isomorphic to
\[
\opn{Tel}(A^0;\mathbf{a}) \otimes_{A^0} B^0 \otimes^{\mrm{L}}_{B^0} B \otimes^{\mrm{L}}_B N.
\]
Our assumption on $\mrm{H}^0(f)$ and (\ref{eqn:RGammaRadical}) imply that
\[
\opn{Tel}(A^0;\mathbf{a}) \otimes_{A^0} B^0 \otimes^{\mrm{L}}_{B^0} B \cong
\opn{Tel}(B^0;\mathbf{b}) \otimes^{\mrm{L}}_{B^0} B,
\]
so the result follows from (\ref{eqn:RGamma}).
\end{proof}

\begin{proof}[Proof of Proposition \ref{prop:ind}]
It follows from Lemma \ref{lem:RGofForg} that there is an equality
\[
\amp\left(\mrm{R}\Gamma_{\bar{\m}}(M)\right) = \amp\left(\mrm{R}\Gamma_{\bar{\m}}(N)\right).
\]
If $M$ is local-Cohen-Macaulay over $A$ then
\[
\amp(M) = \amp(A) = \amp\left(\mrm{R}\Gamma_{\bar{\m}}(M)\right),
\]
and the assumption that $\amp(A) = \amp(B)$ then implies that $M$ is local-Cohen-Macaulay over $B$.
The claim that $M \in \opn{CM}(B) \implies M \in \opn{CM}(A)$ follows similarly.
\end{proof}

Next we discuss the problem of localization.
We begin with a counterexample.

\begin{exa}
Let $(A,\m)$ be a noetherian local ring.
Assume that there is a finitely generated $A$-module $M$ which is maximal Cohen-Macaulay,
such that for some $\p \in \opn{Spec}(A)$ we have that $\p \notin \opn{Supp}(M)$,
and the ring $A_{\p}$ is not Cohen-Macaulay.
As a concrete example,
let $\k$ be a field,
take $A$ to be the localization of $\k[x,y,z]/(y^2z,xyz)$ at the origin,
$M = A/zA$ and $\p = (x,y)$.
It follows from Theorem \ref{thm:triv-ext}
that $B = A \skewtimes M[\dim(A)]$ is a local-Cohen-Macaulay DG-ring.
Localizing at $\p \in \mrm{H}^0(B)$,
the fact that $M_{\p} = 0$ implies that
\[
B_{\p} = \left(A \skewtimes M[\dim(A)]\right)_{\p} \cong A_{\p}
\]
which is not local-Cohen-Macaulay.
\end{exa}

Two features of this counterexample that caused this unfortunate phenomena are that $\amp(B_{\p}) < \amp(B)$,
and that $\opn{Spec}(\mrm{H}^0(B))$ is reducible. 

\begin{thm}\label{thm:localization}
Let $(A,\bar{\m})$ be a noetherian local-Cohen-Macaulay DG-ring,
and let $\bar{\p} \in \opn{Spec}(\mrm{H}^0(A))$.
Suppose that $A$ has a dualizing DG-module.
Assume that $\amp(A_{\bar{\p}}) = \amp(A)$; 
equivalently, that $\bar{\p} \in \opn{Supp}(\mrm{H}^{\inf(A)}(A))$.
Then $\left(A_{\bar{\p}},\bar{\p}\cdot \mrm{H}^0(A_{\bar{\p}})\right)$ is a local-Cohen-Macaulay DG-ring.
\end{thm}
\begin{proof}
Assume $R$ is a dualizing DG-module over $A$.
It follows from \cite[Corollary 6.11]{Sh} that $R_{\bar{\p}}$ is a dualizing DG-module over $A_{\bar{\p}}$.
It follows from the definition of localization that 
\[
\amp(R_{\bar{\p}}) \le \amp(R),
\]
while by Theorem \ref{thm:main-amp} we have that 
\[
\amp(R_{\bar{\p}}) \ge \amp(A_{\bar{\p}}) = \amp(A).
\]
Since $A$ is local-Cohen-Macaulay, 
by Proposition \ref{prop:CM-by-DC} we have that $\amp(A) = \amp(R)$,
so the above inequalities imply that $\amp(R_{\bar{\p}}) = \amp(A_{\bar{\p}})$.
Proposition \ref{prop:CM-by-DC} now implies that $A_{\bar{\p}}$ is local-Cohen-Macaulay.
\end{proof}

\begin{prop}\label{prop:full-supp}
Let $(A,\m)$ be a noetherian local ring such that $\opn{Spec}(A)$ is irreducible;
equivalently, such that $A$ has a unique minimal prime ideal.
Let $M$ be a finitely generated $A$-module such that $\dim(M) = \dim(A)$.
Then $\opn{Supp}(M) = \opn{Spec}(A)$.
\end{prop}
\begin{proof}
Since $M$ is finitely generated by \cite[tag 00L2]{SP},
we have that 
\begin{equation}\label{eqn:suppOffg}
\opn{Supp}(M) = \{\p \in \opn{Spec}(A) \mid \opn{ann}(M) \subseteq \p\}.
\end{equation}
Let $\q$ be the unique minimal prime ideal of $A$.
Since $\dim(M) = \dim(A)$, 
so that $\opn{Supp}(M)$ contains a chain of primes of length $\dim(A)$,
we must have $\q \in \opn{Supp}(M)$.
It follows from (\ref{eqn:suppOffg}) that $\opn{ann}(M) \subseteq \q$,
and since for all $\p \in \opn{Spec}(A)$, 
we have that $\q \subseteq \p$, 
we deduce that $\opn{Supp}(M) = \opn{Spec}(A)$.
\end{proof}

\begin{cor}\label{cor:irr-loc}
Given a noetherian local-Cohen-Macaulay DG-ring $(A,\bar{\m})$,
assume that $A$ has a dualizing DG-module.
If $\opn{Spec}(\mrm{H}^0(A))$ is irreducible,
then for all $\bar{\p} \in \opn{Spec}(\mrm{H}^0(A))$,
the DG-ring $\left(A_{\bar{\p}},\bar{\p}\cdot \mrm{H}^0(A_{\bar{\p}})\right)$ is local-Cohen-Macaulay.
\end{cor}
\begin{proof}
By Proposition \ref{prop:dim-of-inf-cm},
we have that 
\[
\dim\left(\mrm{H}^{\inf(A)}(A)\right) = \dim(\mrm{H}^0(A)),
\]
so by Proposition \ref{prop:full-supp}, 
we deduce that
\[
\opn{Supp}\left(\mrm{H}^{\inf(A)}(A)\right) = \opn{Spec}(\mrm{H}^0(A)).
\]
Hence, for all $\bar{\p} \in \opn{Spec}(\mrm{H}^0(A))$,
we have that $\amp(A_{\bar{\p}}) = \amp(A)$,
so the result follows from Theorem \ref{thm:localization}.
\end{proof}

We now make a global definition of Cohen-Macaulay DG-rings:

\begin{dfn}
Let $A$ be a noetherian DG-ring with bounded cohomology.
We say that $A$ is a Cohen-Macaulay DG-ring if for all $\bar{\p} \in \opn{Spec}(\mrm{H}^0(A))$,
the DG-ring $A_{\bar{\p}}$ is a local-Cohen-Macaulay DG-ring.
\end{dfn}

\begin{prop}
Let $A$ be a noetherian DG-ring which is Gorenstein. Then $A$ is Cohen-Macaulay.
\end{prop}
\begin{proof}
For any $\bar{\p} \in \opn{Spec}(\mrm{H}^0(A))$, 
it follows from \cite[Corollary 6.11]{Sh} that $A_{\bar{\p}}$ is a local Gorenstein DG-ring,
so by Proposition \ref{prop:GorisCM}, $A_{\bar{\p}}$ is local-Cohen-Macaulay.
\end{proof}

\begin{cor}
Let $A$ be a noetherian DG-ring with bounded cohomology,
and assume $A$ has a dualizing DG-module.
Suppose that $\opn{Spec}(\mrm{H}^0(A))$ is locally irreducible;
equivalently, any maximal ideal of $\opn{Spec}(\mrm{H}^0(A))$ contains a unique minimal prime ideal.
Then $A$ is a Cohen-Macaulay DG-ring if and only if for any maximal ideal $\bar{\m} \in \opn{Spec}(\mrm{H}^0(A))$,
the DG-ring $A_{\bar{\m}}$ is local-Cohen-Macaulay.
\end{cor}
\begin{proof}
The assumption that $\opn{Spec}(\mrm{H}^0(A))$ is locally irreducible implies that for any maximal ideal 
$\bar{\m} \in \opn{Spec}(\mrm{H}^0(A))$,
the ring $\mrm{H}^0(A_{\bar{\m}})$ has an irreducible spectrum, 
and by assumption the DG-ring $A_{\bar{\m}}$ is local-Cohen-Macaulay.
Since the property of having a single minimal prime ideal is preserved by localization,
the result now follows from Corollary \ref{cor:irr-loc}.
\end{proof}

\begin{cor}
Let $A$ be a noetherian DG-ring with bounded cohomology,
and assume $A$ has a dualizing DG-module.
Let $n = \amp(A)$, and suppose that the $\mrm{H}^0(A)$-module $\mrm{H}^{-n}(A)$ has full support.
Then $A$ is a Cohen-Macaulay DG-ring if and only if for any maximal ideal $\bar{\m} \in \opn{Spec}(\mrm{H}^0(A))$,
the DG-ring $A_{\bar{\m}}$ is local-Cohen-Macaulay.
\end{cor}
\begin{proof}
This follows from Theorem \ref{thm:localization}.
\end{proof}

We finish this section with the following observation about the cohomology of a dualizing DG-module over
a local-Cohen-Macaulay DG-ring, as it follows from applying localization to it.

\begin{prop}
Let $(A,\bar{\m})$ be a noetherian local-Cohen-Macaulay DG-ring,
and let $R$ be a dualizing DG-module over $A$.
Then
\[
\opn{Supp}\left(\mrm{H}^{\inf(A)}(A)\right) \subseteq \opn{Supp}\left(\mrm{H}^{\inf(R)}(R)\right),
\]
and
\[
\opn{Supp}\left(\mrm{H}^{\inf(A)}(A)\right) \subseteq \opn{Supp}\left(\mrm{H}^{\sup(R)}(R)\right).
\]
\end{prop}
\begin{proof}
Given $\bar{\p} \in \opn{Supp}\left(\mrm{H}^{\inf(A)}(A)\right)$,
we have that $\amp(A_{\bar{\p}}) = \amp(A)$,
so by the proof of Theorem \ref{thm:localization}, 
we have that $\amp(R_{\bar{\p}}) = \amp(R)$.
This in turn implies that
\[
\bar{\p} \in \opn{Supp}\left(\mrm{H}^{\inf(R)}(R)\right),
\]
and that
\[
\bar{\p} \in \opn{Supp}\left(\mrm{H}^{\sup(R)}(R)\right).
\]
\end{proof}

\section{Some remarks on non-negatively graded commutative DG-rings}\label{sec:positive}

In this final section we will work with non-negatively graded commutative DG-rings 
\[
A = \bigoplus_{n = 0}^{\infty} A^n
\]
with a differential of degree $+1$.
We wish to discuss here the question:
when is such a DG-ring Cohen-Macaulay?
We briefly discuss a possible answer to this question.
We will continue to work with the assumption that $A$ is noetherian local and has bounded cohomology;
that is, we assume that $(\mrm{H}^0(A),\bar{\m})$ is a noetherian local ring,
for each $i > 0$, the $\mrm{H}^0(A)$-module $\mrm{H}^i(A)$ is finitely generated,
and for $i>>0$ we have that $\mrm{H}^i(A) = 0$.

Similarly to Section \ref{sec:lc},
the results of \cite{BIK} still apply in the non-negative setting,
and there is a local cohomology functor 
\[
\mrm{R}\Gamma^A_{\bar{\m}}: \cat{D}(A) \to \cat{D}(A).
\]

Note that in the non-negative setting,
the map $\mrm{H}^0(A) \to A$ goes in the other direction,
so there is a forgetful functor $Q:\cat{D}(A) \to \cat{D}(\mrm{H}^0(A))$.
It follows from \cite[Proposition 2.7]{BIK2} that local cohomology commutes with $Q$,
so that there is an isomorphism
\begin{equation}\label{eqn:LCcomNN}
Q \circ \mrm{R}\Gamma^A_{\bar{\m}}(-) \cong \mrm{R}\Gamma^{\bar{A}}_{\bar{\m}}\circ Q(-)
\end{equation}
of functors $\cat{D}(A) \to \cat{D}(\mrm{H}^0(A))$.

This implies that working with local cohomology in the non-negative setting is actually easier,
as the computation immediately reduce to the local ring $\mrm{H}^0(A)$.
This allows us to construct the following counterexample to Theorem \ref{thm:main-amp} in the non-negative setting:

\begin{exa}
Let $(B,\m)$ be a noetherian local Cohen-Macaulay ring of dimension $1$,
and let $\k = B/\m$.
Consider the non-negative noetherian local DG-ring $A = B \skewtimes \k[-1]$.
Then it follows from (\ref{eqn:LCcomNN}) that,
as a complex of $B$-modules
\[
\mrm{R}\Gamma_{\m}(A) \cong \mrm{R}\Gamma_{\m}(B) \oplus \k[-1] 
\]
is concentrated in degree $+1$,
so in particular 
\[
\amp\left(\mrm{R}\Gamma_{\m}(A)\right) = 0 < \amp(A) = 1.
\]
\end{exa}

The reason for the existence of this counterexample is that
\begin{equation}\label{eqn:dim-is-smaller}
\dim\left(\mrm{H}^{\sup(A)}\right) < \dim\left(\mrm{H}^0(A)\right).
\end{equation}
This of course cannot happen in the non-positive setting.

Now, for non-positive local-Cohen-Macaulay DG-rings,
not only that (\ref{eqn:dim-is-smaller}) is an equality,
but also, according to Proposition \ref{prop:dim-of-inf-cm},
there is a dual equality:
\[
\dim\left(\mrm{H}^{\inf(A)}\right) = \dim\left(\mrm{H}^0(A)\right).
\]
If one assumes that the inequality in (\ref{eqn:dim-is-smaller}) is an equality,
then it follows that the analogue of Theorem \ref{thm:main-amp} does hold in the non-negative setting.

In view of this discussion, we propose that a noetherian local non-negative DG-ring $(A,\bar{\m})$
with $\amp(A) < \infty$ will be called local-Cohen-Macaulay if it satisfies the following two conditions:

\begin{enumerate}
\item There is an equality
\[
\dim\left(\mrm{H}^{\sup(A)}(A)\right) = \dim\left(\mrm{H}^0(A)\right).
\]
\item There is an equality
\[
\amp\left(\mrm{R}\Gamma_{\bar{\m}}(A)\right) = \amp(A).
\]
\end{enumerate}

To give further evidence that this is a good definition,
consider a non-negative Gorenstein DG-rings $(A,\bar{\m})$.
It is often the case that if $R$ is a dualizing complex over $\mrm{H}^0(A)$,
then $\mrm{R}\opn{Hom}_{\mrm{H}^0(A)}(A,R)$ is a dualizing DG-module over $A$.
Assuming the uniqueness theorem for dualizing DG-modules holds in this setting,
the assumption that $A$ is Gorenstein implies that $A$ is isomorphic to a shift of 
$\mrm{R}\opn{Hom}_{\mrm{H}^0(A)}(A,R)$.
Then Grothendieck's local duality implies that
\[
\amp\left(\mrm{R}\Gamma_{\bar{\m}}(A)\right) = \amp(A).
\]
Moreover, a calculation shows that $\depth(A) = \dim(\mrm{H}^0(A))$,
and from this one can deduce that
\[
\dim\left(\mrm{H}^{\sup(A)}(A)\right) = \dim\left(\mrm{H}^0(A)\right),
\]
so that $A$ is local-Cohen-Macaulay in the above sense.

\textbf{Acknowledgments.}

The author is grateful for the anonymous referee for many comments that helped improving the manuscript.
The author thanks Amnon Yekutieli for many useful suggestions.
The author was partially supported by the Israel Science Foundation (grant no. 1346/15).
This work has been supported by Charles University Research Centre program No.UNCE/SCI/022.

\end{document}